\documentclass[final]{siamltex}
\usepackage{amssymb}
\usepackage{graphicx,fancyhdr}
\usepackage{multirow,amsmath}
\usepackage{dsfont}
\usepackage{subfigure}
\usepackage{bbm}
\usepackage{booktabs}
\usepackage{bbm}
\usepackage{amsfonts}
\usepackage{mathrsfs,float}
\usepackage{color,bm}
\makeatletter

\newtheorem{thm}{Theorem}[section]

\newtheorem{rem}[thm]{Remark}
\usepackage[pagewise]{lineno}

\title{ A Local Fourier Extension Method for Function  Approximation\thanks{The research is partially supported by National Natural Science Foundation of China (Nos. 12171455, RSF-NSFC 23-41-00002)}}

\author{Zhenyu Zhao\thanks{School of Mathematics and Statistics, Shandong University of Technology, Zibo, 255049, China,({\tt Zhenyu\_Zhao@sdut.edu.cn}).} \and Yanfei Wang \thanks{Key Laboratory of Deep Petroleum Intelligent Exploration and Development, Institute of Geology and Geophysics, Chinese Academy of Sciences, Beijing, 100029,China,({\tt yfwang@mail.iggcas.ac.cn}).}}

\begin{document}

\maketitle

\begin{abstract}
This paper proposes a novel localized Fourier extension method for approximating non-periodic functions via domain segmentation. By partitioning the computational domain into subregions with uniform discretization scales, the method achieves spectral accuracy at $\mathcal{O}(M)$ computational complexity. Theoretical error bounds and parameter dependency analyses validate the robustness of the proposed method. The relationship among the key parameters involved is analytically established, accompanied by an optimized parameter selection strategy. Numerical experiments further confirm the effectiveness of the proposed method.
\end{abstract}



\pagestyle{myheadings}
\thispagestyle{plain}
\markboth{Z. Y. ZHAO and Y.F. Wang}{ Local Fourier Extension Method}
\section{Introduction}
Numerical approximation methods constitute a cornerstone of computational mathematics, playing a pivotal role in both theoretical advancements and engineering applications. Classical approximation theories not only provide essential tools for constructing numerical solutions to differential equations but also continuously drive technological innovations across multiple disciplines. For instance, the high-precision characteristics of spectral methods are fundamentally rooted in comprehensive studies of polynomial approximation and Fourier series expansion\cite{shen2011spectral}. Meanwhile, wavelet multiscale analysis, through its breakthrough in time-frequency localization properties, has established rigorous mathematical foundations for practical implementations such as image compression and signal  denoising \cite{cohen2011wavelets}.

In recent decades, frame approximation theory has emerged as a paradigm-shifting development in function representation, attracting significant attention due to its enhanced flexibility beyond conventional basis constraints \cite{adcock2019frames,herremans2024efficient,adcock2020frames}. This framework employs redundant representation systems while maintaining stability, thereby demonstrating superior adaptability to complex functional features. Of particular interest is the Fourier extension technique--a prototypical frame approximation approach that effectively mitigates Gibbs phenomena in classical Fourier approximations by extending target functions to larger periodic domains. This methodology has opened novel pathways for handling non-periodic functions while preserving spectral accuracy \cite{Huybrechs2010,bruno2010high,plonka2018numerical,Bruno2022,zhao2025new}.

The convergence properties of the Fourier extension method are comprehensively discussed in \cite{adcock2019frames, adcock2014numerical, Huybrechs2010, Webb2020}. Regarding numerical implementation,  remarkable progress has been made through the analysis of the discrete matrix's singular value distribution combined with randomized algorithms. Notably, \cite{Lyon2011} and \cite{Matthysen2016FAST} developed fast algorithms achieving $\mathcal{O}(M\log^2 M)$ complexity, where
$M$ represents the number of nodal points. Building upon these foundations, \cite{zhao2024fast} presents a more efficient approach that utilizes only boundary data for extension operations combined with FFT-based global approximation, attaining $\mathcal{O}(M\log M)$ complexity.

While existing methods predominantly adopt a global approximation perspective, they suffer from two fundamental limitations. First, localized singularities can severely degrade approximation accuracy across the entire computational domain. Second, the direct application of discretization techniques--such as collocation or Galerkin methods--typically results in fully populated discrete matrices. A key insight from \cite{zhao2024fast} is that increasing the extension length can reduce the number of required boundary nodes. In this work, we extend this idea to the interior of the domain through a domain decomposition strategy.

This paper proposes a partitioned extension methodology where the computational domain is subdivided into multiple subregions for localized extension operations. By employing uniform discretization scales across subdomains, we enable matrix reuse and reduce overall computational complexity to $O(M)$ while maintaining approximation quality. This structured approach not only enhances computational efficiency but also naturally lends itself to parallel implementation architectures.

\subsection{Overview of the paper}
The paper proceeds with the following organizational framework: Section \ref{SEC2} develops the methodological architecture, presenting both the theoretical framework and associated error analysis through rigorous mathematical proofs. Following this theoretical foundation, Section \ref{SEC3} investigates the computational implementation through  systematic parameter investigations and present practical guidelines for parameter selection derived from computational complexity analysis. The subsequent Section \ref{SEC4} validates the proposed methodology through comprehensive numerical experiments. The final section synthesizes key findings, discusses the method's broader implications for computational mathematics, and proposes directions for future research.
\section{The local Fourier extension method\label{SEC2}}
\subsection{ Outline of the method}
Let  $f\in C^{\infty}[a,b]$, points $a=a_0< a_1< \ldots <a_K=b$ divide $I=[a,b]$ into $K$ subintervals $I_1,I_2,\ldots, I_K$. Let  $f_k(x)$ be a function defined on interval $I_k$, satisfying:
\begin{equation}
  f_k(x)=f(x),\quad x\in I_k,\quad k=1,2,\ldots, K.
\end{equation}
Let $\Omega=[0,2\pi]$, $\Lambda=[0,\frac{2\pi}{T}]
\,\, (T>1)$.  $\{g_k(t)\}_{k=1}^K$ is a family of functions defined on $\Lambda$, such that
\begin{equation}
  g_k(t)=f_k(a_{k-1}+s_k t),\quad \forall t\in \Lambda,
\end{equation}
where
\begin{equation}
  s_k=\frac{T}{2\pi}(a_k-a_{k-1}).
\end{equation}
Let
\begin{equation}
\phi_{\ell}(t)=e^{i\ell t},\quad\ell=\pm1,\pm 2,\ldots,\quad t\in\Lambda.
\end{equation}
For any fixed $N$, we define
\begin{equation}
  \Phi_N=\{\phi_\ell\}_{|\ell|\leq N},
\end{equation}
It is  a frame for its span $\text{span}(\Phi_N)=:H_N$.
Define the operators
\begin{equation}
  \begin{aligned}
&{\mathcal{F}}_N:\mathbb{C}^{2N+1}\rightarrow H_N,& {\bf{c}}_N=\{ {c}_{\ell}\}_{|\ell|\leq N}\mapsto \sum_{\ell=-N}^{N}   {c}_{\ell}\phi_{\ell}.
  \end{aligned}
\end{equation}
Since ${\mathcal{F}}_N$ is of finite rank, it is a compact operator. Let $(\sigma_j,v_j,u_j)$ be a singular system of ${\mathcal{F}}_N$  and $\epsilon>0$ be a tolerance, then the  truncated singular value decomposition (TSVD) solution of equation
\begin{equation}\label{continuexteneq}
{\mathcal{F}}_N {\bf c}_{N}=g
\end{equation}
can be given as
\begin{equation}
  {\bf c}_{N}^{\epsilon}=\sum_{\sigma_j>\epsilon}\frac{\langle g, u_j\rangle}{\sigma_j}v_j.
\end{equation}
And we define the operator $\mathcal{Q}_{N}^{\epsilon}:L^2(\Lambda)\rightarrow \tilde{H}=:\text{ span}\{{u}_j:{\sigma}_j>\epsilon\}$  as
\begin{equation}
\mathcal{Q}_{N}^{\epsilon}g:={\mathcal{F}}_N{\bf c}_{N}^{\epsilon}=\sum_{\sigma_j>\epsilon}{\langle g, u_j\rangle}u_j.
\end{equation}
Thus, $\mathcal{Q}_{N}^{\epsilon}$ is the orthogonal projection from $L^2(\Lambda)$ to $\tilde{H}$.

Now we use
\begin{equation}
\left(\mathcal{P}_{N,K}^{\epsilon}f\right)(x):=\left(\mathcal{Q}_{N}^{\epsilon}g_k\right)\left(\frac{x-a_{k-1}}{s_k}\right),\quad x\in I_k, \quad k=1,2,\ldots, K
\end{equation}
as the local Fourier extension approximation of the function $f$.
\subsection{Error estimate of the method}
We can obtain the following error estimates:
\begin{lemma}The Fourier extension approximation $\mathcal{Q}_{N}^{\epsilon}g_k$ satisfies
  \begin{equation}\label{est1}
  \left\|g_k-\mathcal{Q}_{N}^{\epsilon}g_k\right\|_{L^2(\Lambda)}\leq \inf \{\|g_k-\mathcal{F}_N {\bf c}_N\|_{L^2(\Lambda)}+\epsilon\|{\bf c}_N\|_{\ell^2}:~{\bf c}_N\in\mathbb{C}^{2N+1}\}.
  \end{equation}
\end{lemma}
\begin{proof}For any ${\bf c}_N\in \mathbb{C}^{2N+1}$, note that
 { $$\mathcal{Q}_{N}^{\epsilon}({\mathcal{F}}_N{\bf c}_N)\in \tilde{H}$$
 and $\mathcal{Q}_{N}^{\epsilon}g_k$ is the orthogonal projection of $g_k$ in $\tilde{H}$. So we can obtain
     \begin{equation*}
    \begin{aligned}
      \|g_k-\mathcal{Q}_{N}^{\epsilon}g_k\|_{L^2(\Lambda)}&\leq &&  \|g_k-\mathcal{Q}_{N}^{\epsilon}({\mathcal{F}}_N{\bf c}_N)\|_{L^2(\Lambda)} \\
      &\leq && \|g_k-\mathcal{F}_N{\bf c}_N\|_{L^2(\Lambda)}+\|\mathcal{F}_N{\bf c}_N-\mathcal{Q}_{N}^{\epsilon}({\mathcal{F}}_N{\bf c}_N)\|_{L^2(\Lambda)}.
    \end{aligned}
  \end{equation*}}
Moreover, since $\mathcal{F}_N{\bf c}_N\in H_N=R({\mathcal{F}}_N)$, we have
\begin{equation}
\begin{aligned}
\|\mathcal{F}_N{\bf c}_N-\mathcal{Q}_{N}^{\epsilon}{\mathcal{F}}_N{\bf c}_N\|_{L^2(\Lambda)}&=&&\left\|\sum_{\sigma_j\leq \epsilon}{\langle \mathcal{F}_N{\bf c}_N,u_j\rangle}u_j\right\|_{L^2(\Lambda)}\\
&=&&\left\|\sum_{\sigma_j\leq \epsilon}\sigma_j{\langle{\bf c}_N,v_j\rangle}u_j\right\|_{L^2(\Lambda)}\\
&\leq&&\epsilon\left\|{\bf c}_N\right\|_{\ell^2}
\end{aligned}
\end{equation}
This  completes the proof.
\end{proof}
\begin{theorem}
 The local Fourier extension approximation $\mathcal{P}_{N,K}^{\epsilon}f$ satisfies
 \begin{equation}\label{est2}
   \left\|f-\mathcal{P}_{N,K}^{\epsilon}f\right\|\leq \sqrt{\frac{T}{2\pi}(b-a)} \max_{1\leq k\leq K}\left\{\left\|g_k-\mathcal{Q}_{N}^{\epsilon}g_k\right\|_{L^2(\Lambda)}\right\}.
 \end{equation}
\end{theorem}
\begin{proof}
 \begin{equation}
\begin{aligned}
  \|f-\mathcal{P}_{N,K}^{\epsilon}f\|^2_{L^2(I)}&=\int_{a}^b\left|f(x)-\left(\mathcal{P}_{N,K}^{\epsilon}f\right)(x)\right|^2dx\\
  &=\sum_{k=1}^K\int_{a_{k-1}}^{a_k}\left|f_k(x)-\left(\mathcal{Q}_{N}^{\epsilon}g_k\right)\left(\frac{x-a_{k-1}}{s_k}\right)\right|^2dx\\
  &=\sum_{k=1}^K\int_{\Lambda}s_k\left|g_k(t)-\left(\mathcal{Q}_{N}^{\epsilon}g_k\right)(t)\right|^2dt\\
  &\leq\frac{T}{2\pi} \sum_{k=1}^K(a_k-a_{k-1})\left\|g_k^c-g_k\right\|_{L^2(\Lambda)}^2.
\end{aligned}
 \end{equation}
\end{proof}

\begin{rem}
  \begin{enumerate}
  \item We can see that for the function $f(x)=e(\text{i}\pi\omega x)$, its corresponding
  \begin{equation}\label{omegachange}
 g_k(t)=e^{\text{i}\omega\pi a_{k-1}}e^{\text{i}\omega\pi s_k t}.
  \end{equation}
  That is, the frequency of the function changes from $\omega\pi$ to $s_k\omega\pi$. Therefore, for high-frequency functions, as long as the interval division can ensure that the selected $a_k-a_{k-1}$ is small enough, it can play a role in reducing the frequency in the extension calculation, so that only a small $N$ is needed to achieve high-precision approximation.
  \item Combining the results of \eqref{est1} and \eqref{est2}, we can see that whether the method can obtain a good approximation result depends mainly on whether the extended approximation on each subinterval can achieve the desired accuracy. According to the  approximation theory, for low-frequency functions $g(t)$,  there always exists an appropriate value of $N$ that enables the vector ${\bf c}_N$ to be found so that $\|g-\mathcal{F}_N{\bf c}_N\|_{L^2(\Lambda)}=\mathcal{O}(\epsilon)$ and $\|{\bf c}_N\|_{\ell^2}$ is not too large.
\end{enumerate}

\end{rem}

\section{Numerical implementation and parameter determination\label{SEC3}}In specific problems, the lengths of the subintervals $I_k$ are not necessarily the same. When the frequency of the function {\color{red} $f$} varies greatly over the entire interval $[a, b]$, it is a better choice to set the interval length according to the frequency. This can be achieved by testing the approximation error a posteriori. In order to ensure the consistency of the discrete matrix in the subsequent extension calculation process, we take an equal number of  equidistant nodes in each subinterval. That is, we take $K$ groups of nodes $\{x_{ki}\}_{i=1}^m$, $k=1,2,\ldots,K$:
\begin{equation}
  x_{ki}=a_{k-1}+i h_k,\quad h_k=\frac{a_{k}-a_{k-1}}{m-1}, \quad i=0,1,\ldots m-1.
\end{equation}
\subsection{Discretization of equation \eqref{continuexteneq}}
Take the sampling ratio $\gamma\geq1$ and let
\begin{equation*}
m=\lceil\gamma(2N+1)\rceil,
\end{equation*}
  where $\lceil\cdot\rceil$ is the rounding symbol.
 After selecting the parameters $N,\gamma, T$ and let
 \begin{equation}
  L=\lceil T\times m \rceil,\quad h=\frac{2\pi}{L},\quad t_i=i h, \quad i=0,1,\ldots,m-1.
  \end{equation}
  Then we get the discrete form of equation \eqref{continuexteneq}:
\begin{equation}\label{discreteeq}
{\bf F}_{\gamma,N}^{T}  {\bf c}_{N,k} ={\bf g}_k,
  \end{equation}
  where
  \begin{equation}
     \left({\bf F}_{\gamma,N}^{T}\right)_{i,\ell}=\frac{1}{\sqrt{L}}\phi_{\ell}(t_{i}),\quad \left({\bf g}_k\right)_{i}=\frac{1}{\sqrt{m}}g_k(t_{i})=\frac{1}{\sqrt{m}}f_k(x_{ki}).
  \end{equation}
  The SVD of ${\bf F}_{\gamma,N}^{T,p}$ can be given as
\begin{equation}\label{svdF}
 {\bf F}_{\gamma,N}^{T}=\sum_{j=1}^{2N+1}{\bf u}_j {\bm \sigma}_j{\bf v}^T_j,
\end{equation}
where ${\bm \sigma}_j$ are the singular values of ${\bf F}_{\gamma,N}^{T}$ such that
\begin{equation*}
  {\bm\sigma}_1\geq \ldots\geq  {\bm\sigma}_{2N+1}\geq 0,
\end{equation*}
while ${\bf u}_i$ and ${\bf v}_i$ are the left and right singular vectors of ${\bf F}_{\gamma,N}^{T}$, respectively.

Now  the TSVD solution of \eqref{discreteeq} can be given as
\begin{equation}\label{GTSVDsoldiscrete}
{\bf c}_{N,k}^{\epsilon}=\sum_{{\bm\sigma}_j>\epsilon}\frac{\langle {\bf g}_k, {\bf u}_j \rangle}{{\bm \sigma}_j}{{\bf v}}_j
\end{equation}
and we obtain the discrete approximation of  ${\mathcal{Q}}_{N}^{\epsilon}g_k$:
\begin{equation}\label{coefandQ1}
 {\bf Q}_{N}^{\epsilon}{\bf g}_k={\bf F}_{\gamma,N}^{T}{\bf c}_{N,k}^{\epsilon}.
\end{equation}
\subsection{Numerical testing  of key parameters}
Through  simple analysis, we can preliminarily see the impact of various parameters on the performance of the algorithm: Parameter $m$ governs the computational complexity during sub-region extension, while parameter $K$ controls the frequency characteristics of function $g_k(t)$, thereby determining the magnitude of $N$. The product $T\gamma$ dictates the discrete step size over interval $[0, 2\pi]$, which critically influences algorithm convergence. Furthermore, the  product $Km$ collectively determines the total node count across the operational domain $[a,b]$. Next, we will systematically examine these parameter relationships through targeted numerical experiments.

We take $a=-1, b=1$ and investigate the configuration with uniformly distributed partition points $a_k$. Specifically,  let $M=K(m-1), h=\frac{2}{M},$ and the nodal coordinates are constructed as:
$$ x_{ki}=-1+(k-1)(m-1)h+ih,\quad k=1,2,\ldots,K,\quad i=0,1,\ldots,m-1.$$
For numerical testing, we adopt the oscillatory test function:
$$f(x)=e^{\text{i}\pi\omega x}.$$
We first quantify the interplay between parameter $T$ and other parameters. Throughout our experiments, approximation errors are computed as the maximum pointwise error across a refined evaluation grid ($10\times$ denser than the construction grid), with the regularization parameter fixed at $\epsilon = 1e -14$.
\begin{figure}
			\begin{center}
\subfigure[\label{4a}$\gamma=2$, $K=20$, $N=20$] {
\resizebox*{5.5cm}{!}{\includegraphics{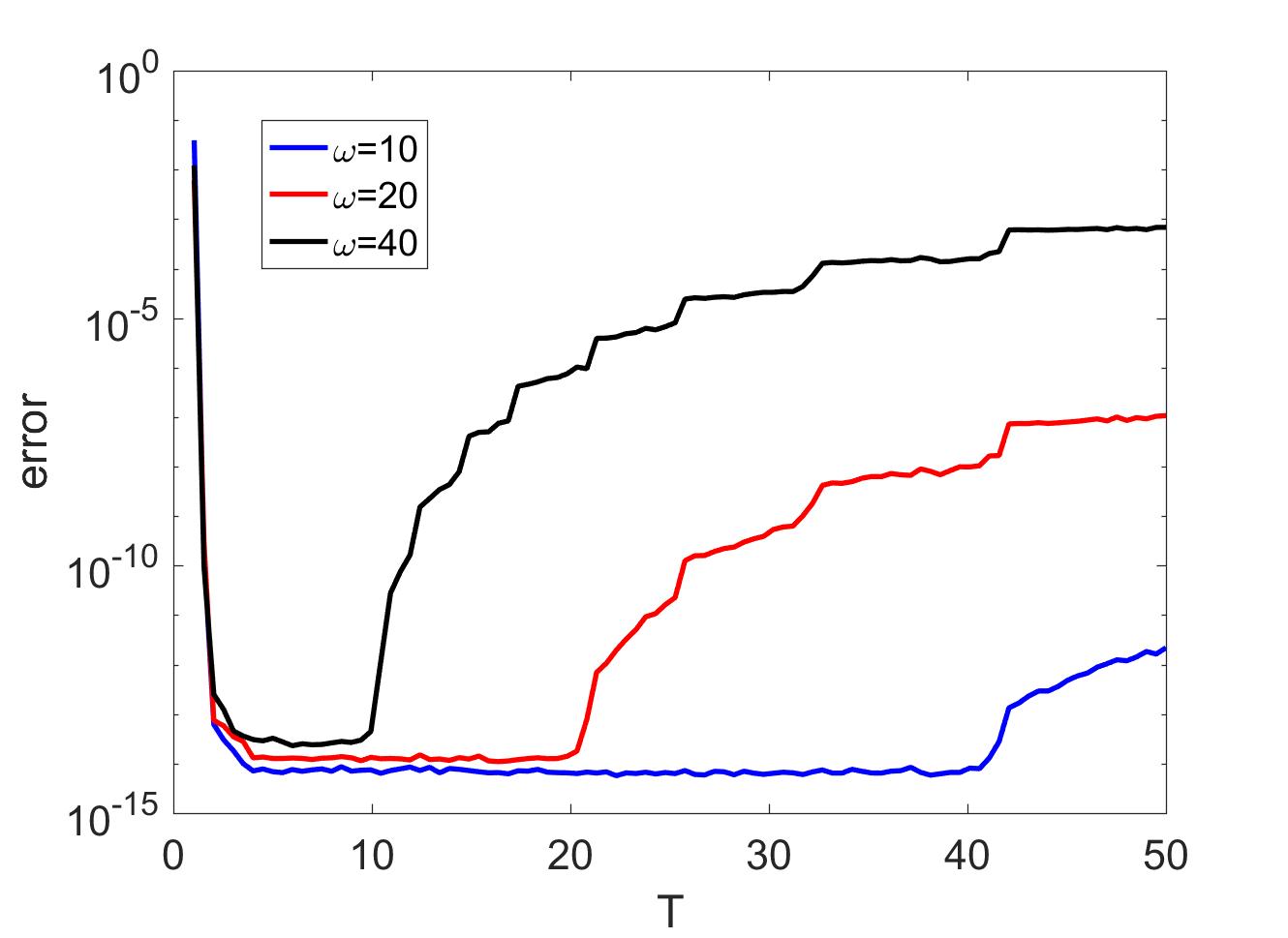}}
}%
\subfigure[\label{4b}$\omega=20$, $K=20$,  $N=20$] {
\resizebox*{5.5cm}{!}{\includegraphics{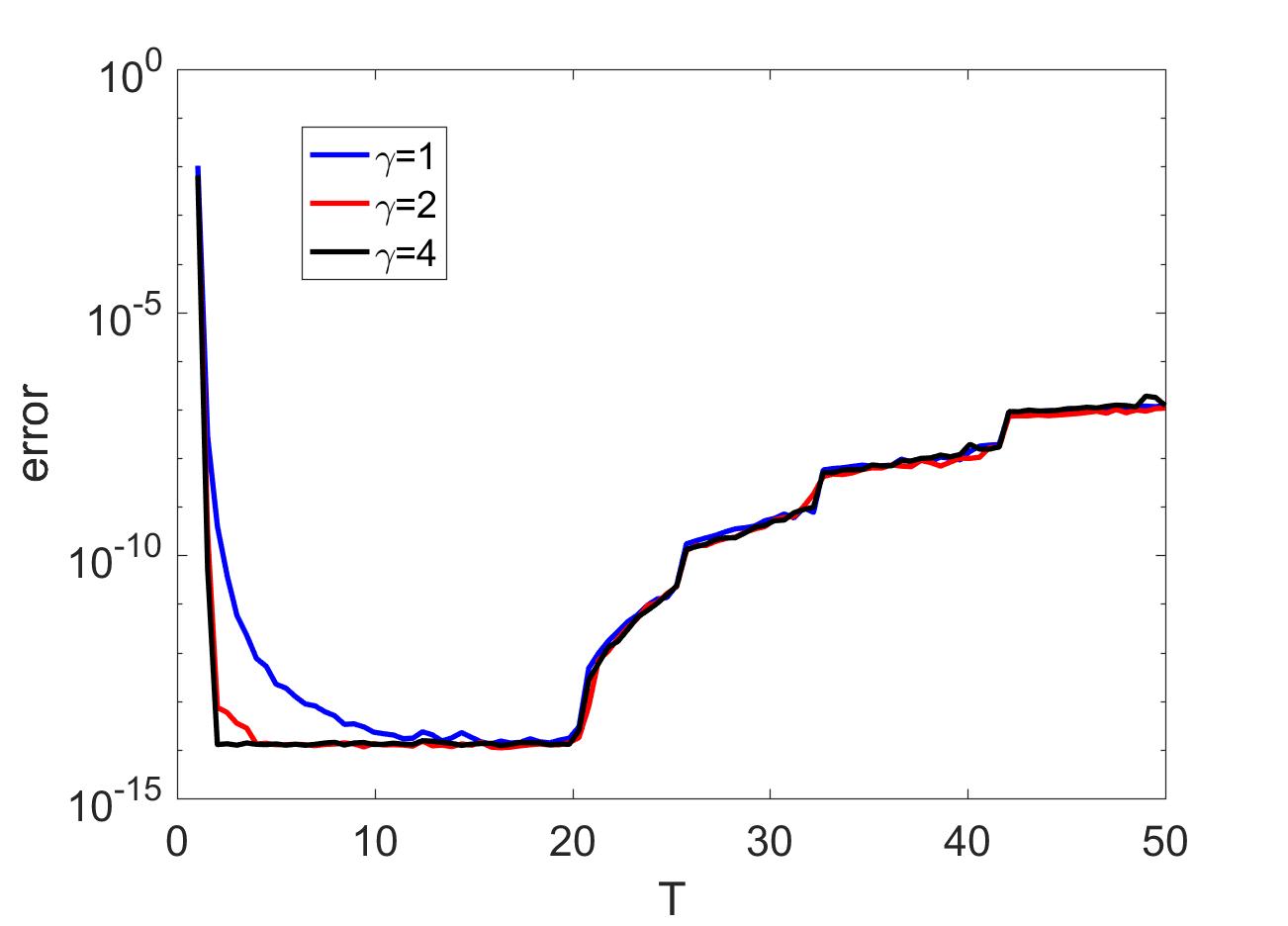}}
}%

 \subfigure[\label{4c}$\omega=20$, $\gamma=2$, $N=20$] {
\resizebox*{5.5cm}{!}{\includegraphics{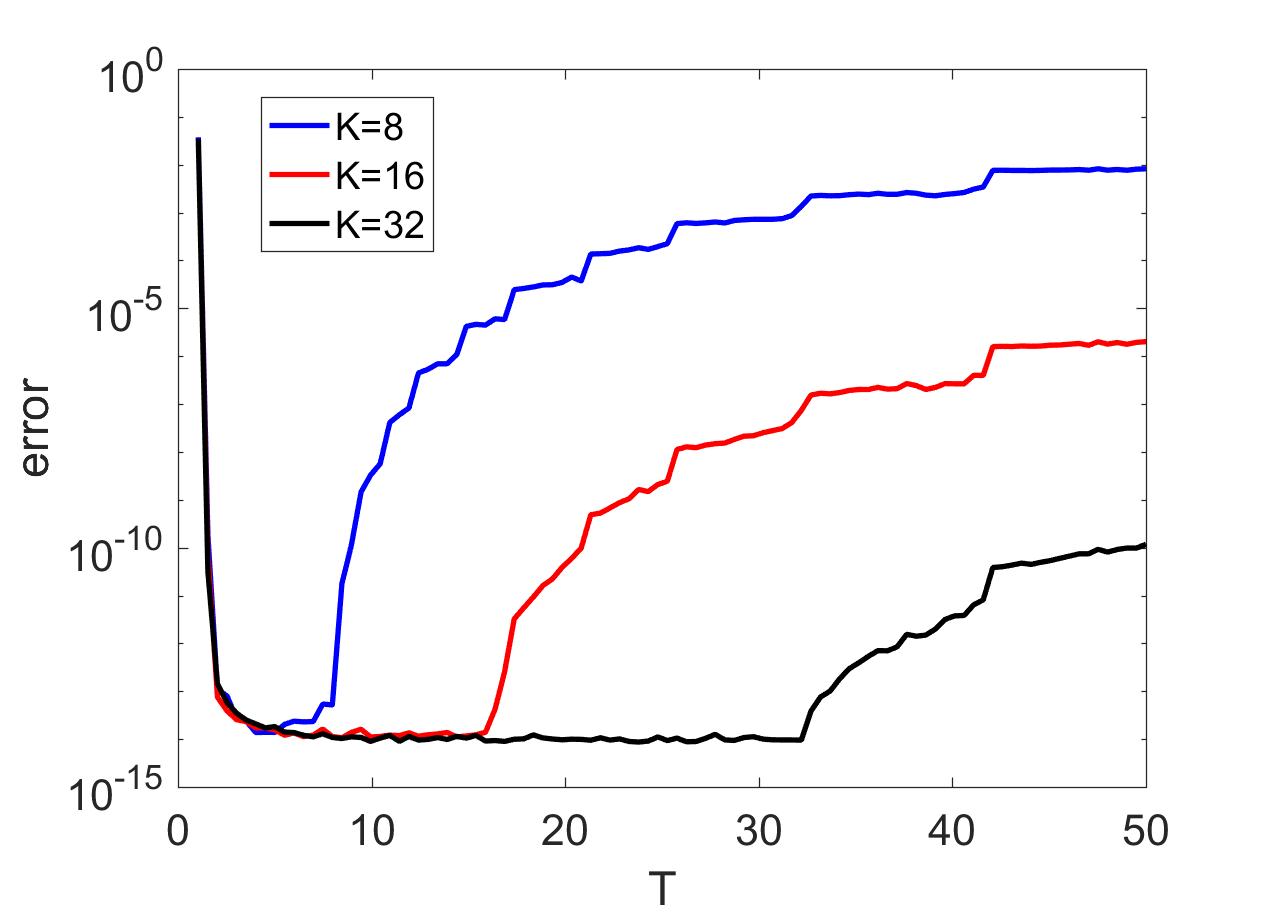}}
}%
 \subfigure[\label{4d}$\omega=40$, $\gamma=2$,$K=20$] {
\resizebox*{5.5cm}{!}{\includegraphics{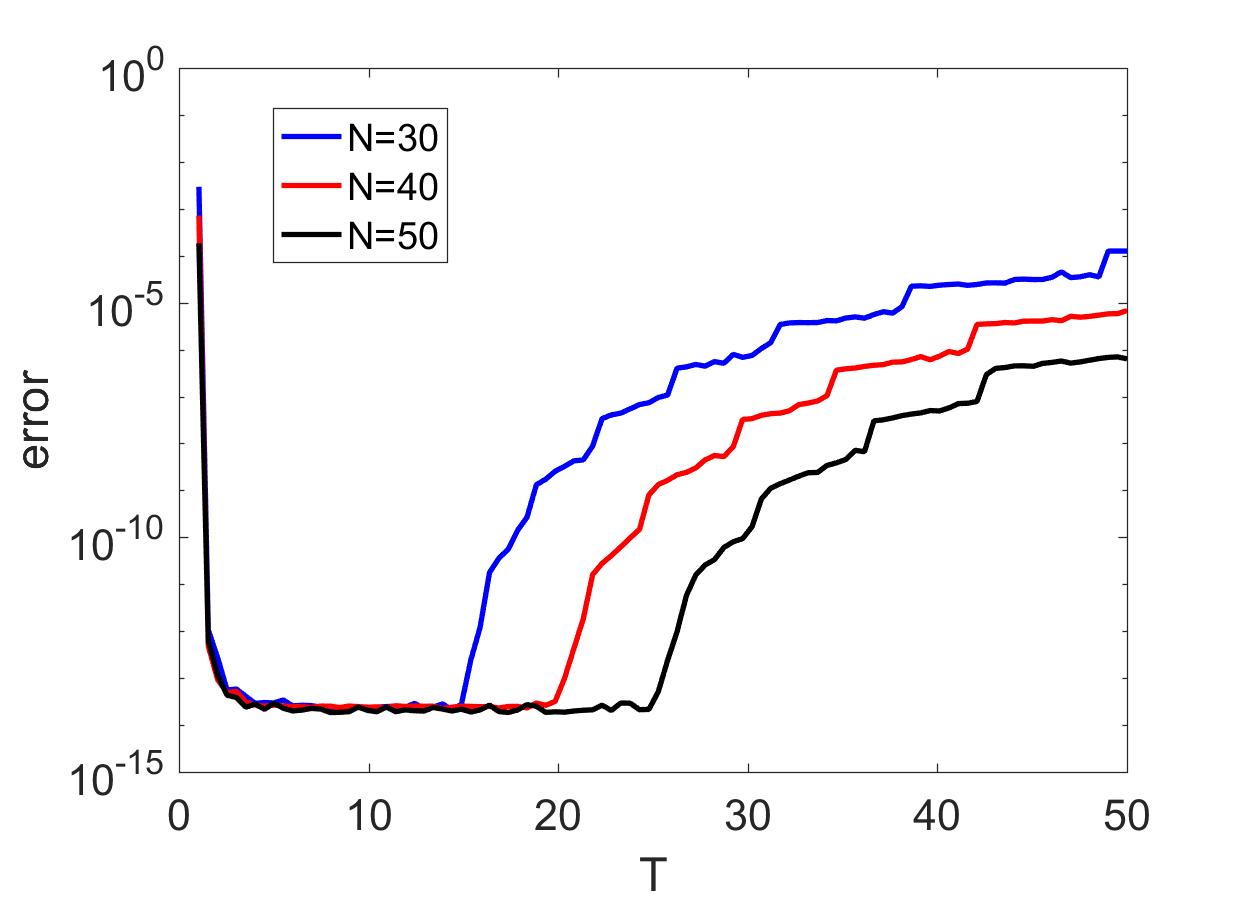}}
}%
		{\caption{The approximation error  against $T$ for different values of $\omega$,  $\gamma$, $N$ and $K$. \label{figTchange}}}
	\end{center}
\end{figure}

As demonstrated in  Fig. \ref{figTchange}, the proposed method achieves machine precision exclusively when the parameter
$T$  resides within a specific interval $[T_1,T_2]$. The lower bound $T_1$
 is solely determined by the parameter $\gamma$, whereas the upper bound $T_2$ exhibits dependence on multiple parameters:
$\omega$, $N$, and $K$. This behavior originates from the fundamental relationship $T\gamma$, which governs the discrete step size over the interval $[0,2\pi]$, thereby affecting the approximation of the discrete matrix $ {\bf F}_{\gamma,N}^{T}$ to the operator ${\mathcal{F}}_N$, and consequently dictates the convergence characteristics within each subinterval.
Fig. \ref{4b} clearly illustrates the inverse proportionality between $T_1$  and $\gamma$, showing decreased
$T_1$  values with increasing $\gamma$. Furthermore, as established in equation \eqref{omegachange}, effective approximation of
$g_k(t)$  across subintervals requires satisfaction of the condition:
\begin{equation}\label{relationN}
 N \geq \frac{\omega T}{K}.
\end{equation}
This relationship yields the upper bound estimate:
\begin{equation}
T_2 \leq \frac{KN}{\omega},
\end{equation}
which aligns consistently with the numerical results presented in the graphs. These results also reflect the possibility of our expectation to calculate with a fixed $N$ value. Specifically, for high-frequency functions, the validity of Equation \eqref{relationN} under constant $N$ can be preserved through strategic increases in parameter $K$. Equation \eqref{relationN} also reveals a critical relationship: expansion of $T$ proportionally necessitates larger $K$ values, corresponding to finer interval partitioning. Therefore, we are more concerned about the value of $T_1$ corresponding to different $\gamma.$ We tested this and the results are shown in Table \ref{tab1}.
  \begin{table}
\begin{center}
  \caption{The approximation value of ${T}_{1}$  for various $\gamma$. \label{tab1} }
  \small
{\begin{tabular*}{\textwidth}{@{\extracolsep\fill}cccccccccccccc} \toprule
&$\gamma=1$&$\gamma=1.2$&$\gamma=1.5$&$\gamma=2$&$\gamma=3$&$\gamma=4$\\\midrule
$T_1$&$5.6$&$4.5$&$3.9$&$2.3$&$1.6$&$ 1.2$ \\
\bottomrule
  \end{tabular*}}
\end{center}
\end{table}

We now establish the minimum parameter magnitude $N$ required to attain spectral accuracy in general low-frequency function approximations. We conduct numerical experiments with three representative frequencies: $\omega\in\{\sqrt{2},e^2,\sqrt{2\pi}\}$.

 Fig. \ref{figNchange} reveals two fundamental convergence characteristics:
\begin{enumerate}
  \item Exponential Error Decay: The approximation error diminishes super-algebraically with increasing $N$, achieving spectral accuracy upon reaching a critical threshold $N_{0}$.
  \item Parameter Dependence: The convergence threshold $N_{0}$  exhibits:
\begin{itemize}
  \item Exclusive dependence on $T$.  As $T$ increases, the value of $N_0$ decreases.
  \item Immunity to $K$ variations, despite $K$'s role in frequency modulation.
  \end{itemize}
 \end{enumerate}
 \begin{table}
\begin{center}
  \caption{The approximation value of ${N}_{0}$  for various $T$. \label{tab2} }
  \small
{\begin{tabular*}{\textwidth}{@{\extracolsep\fill}cccccccccccccc} \toprule
&$T=1.1$&$T=1.5$&$T=2$&$T=4$&$T=6$&$T=10$&$T=15$\\\midrule
$N_0$&$78$&$29$&$18$&$ 10$ &$9$&$7$&$5$\\
\bottomrule
  \end{tabular*}}
\end{center}
\end{table}

\begin{figure}
			\begin{center}
\subfigure[\label{5a}$\gamma=2$, $T=4$, $K=20$] {
\resizebox*{5.5cm}{!}{\includegraphics{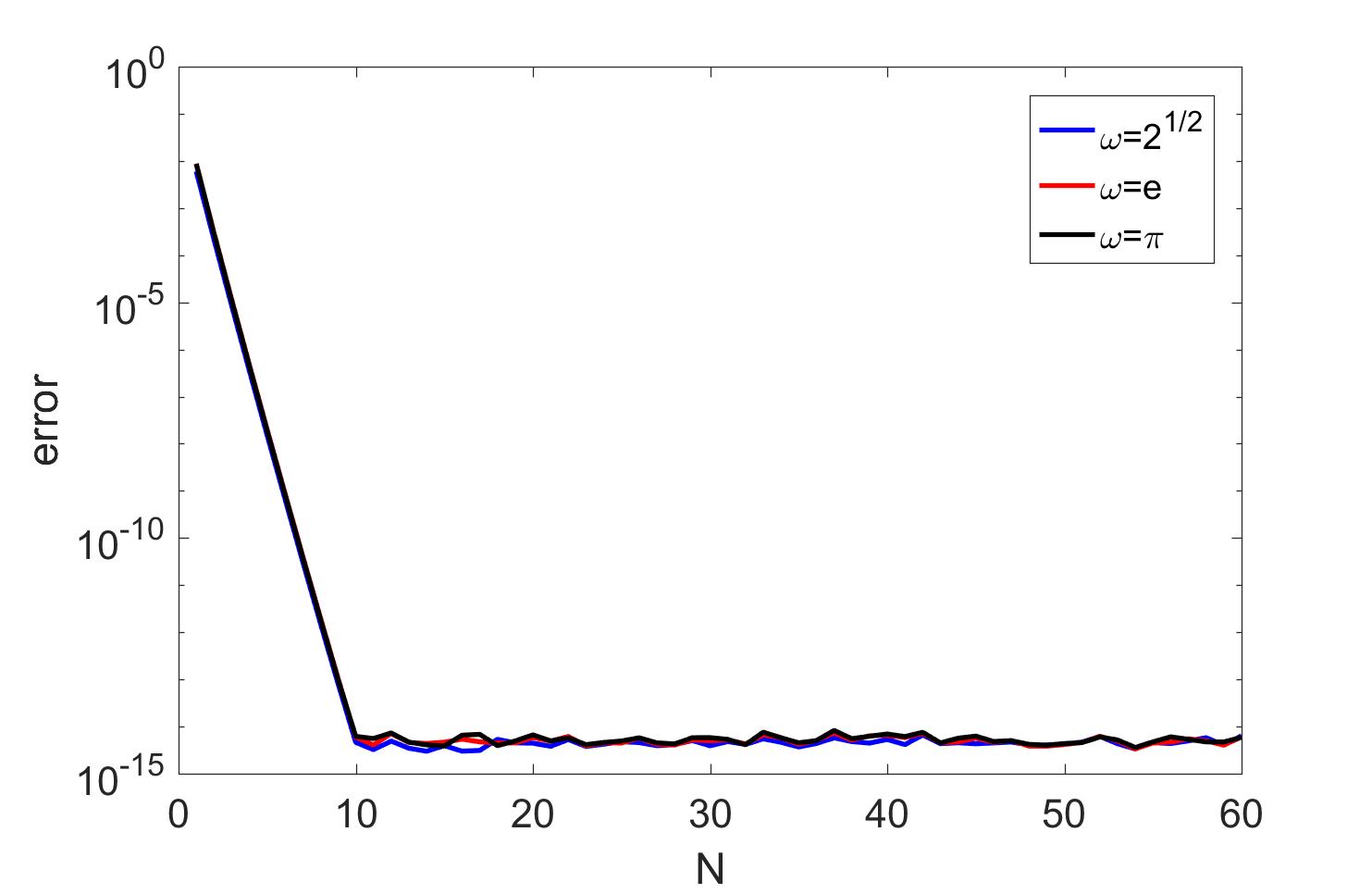}}
}%
\subfigure[\label{5b}$\gamma=4$, $K=20$, $\omega=\sqrt{2}$] {
\resizebox*{5.5cm}{!}{\includegraphics{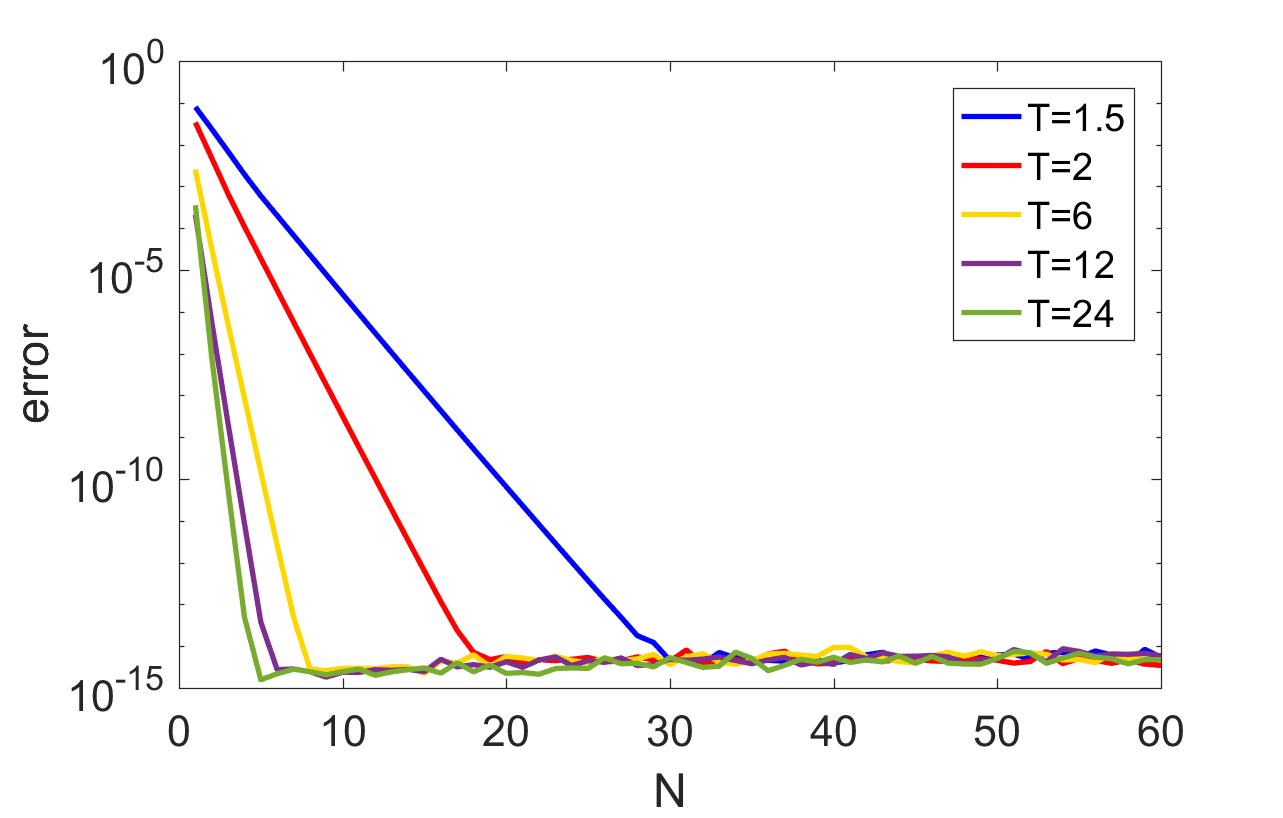}}
}%

 \subfigure[\label{5c} $T=12$, $K=20$, $\omega=e$] {
\resizebox*{5.5cm}{!}{\includegraphics{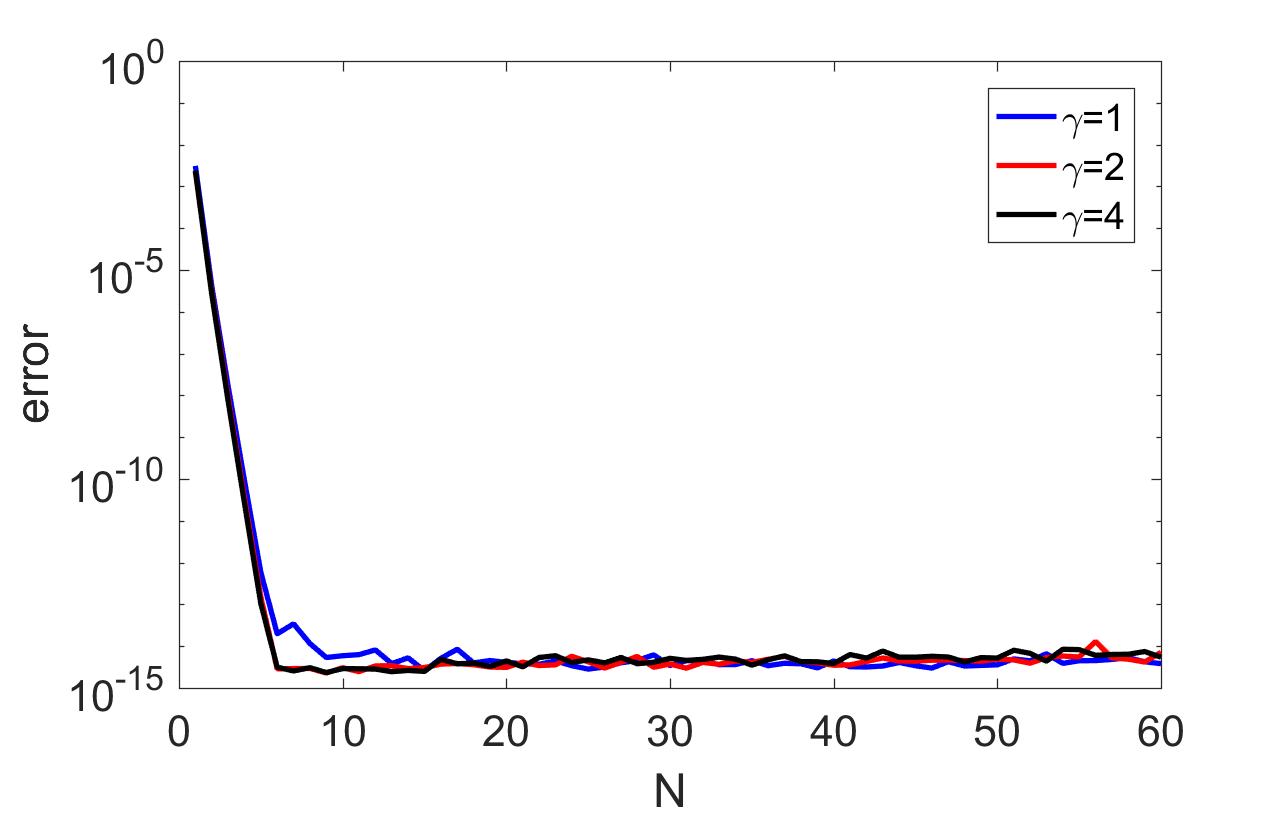}}
}%
 \subfigure[\label{5d}$\gamma=2$, $T=4$, $\omega=\pi$]  {
\resizebox*{5.5cm}{!}{\includegraphics{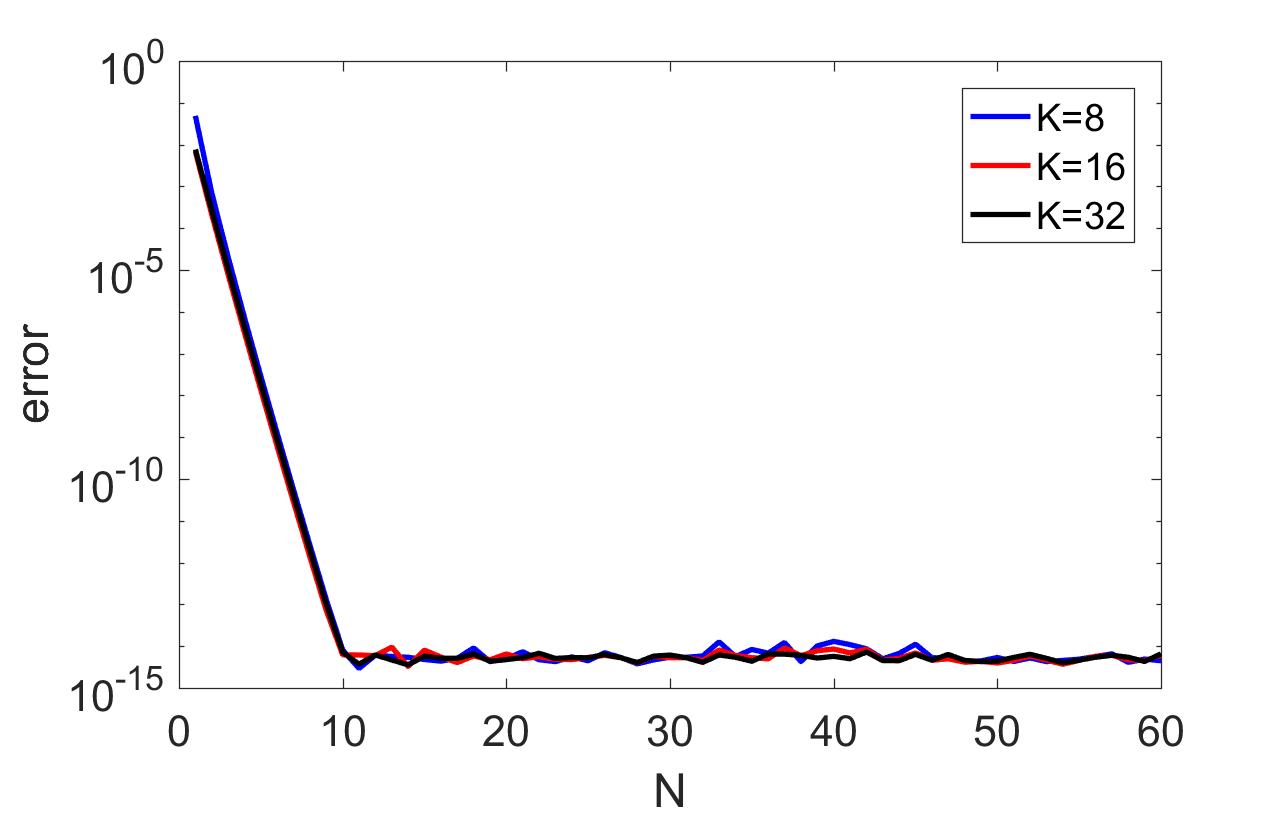}}
}%
		{\caption{The approximation error  against $N$ for different values of $\omega$,  $\gamma$, $T$ and $K$. \label{figNchange}}}
	\end{center}
\end{figure}
\subsection{Computational complexity analysis and parameter determination}
According to the test results in the previous section, the algorithm in this paper can be calculated using a fixed $ N_{\Delta}\geq N_0$. The overall algorithm flow can be described as follows:
1. Given the split point ${a_k}_{k=1}^K$, parameters $T$, $\gamma$, $N_{\Delta}$, and $\epsilon$.
2. Form the discrete matrix ${\bf F}_{\gamma, N_{\Delta}}^T$ and perform SVD decomposition on it, select ${\bm \sigma}_j\geq \epsilon$ and the corresponding ${\bf u}_j, {\bf v}_j$.
3. For $k=1,2,\ldots K,$ calculate $\tilde{\bf c}^{\epsilon}_{N_\Delta,k}$.

Under this setting, the number of nodes in each subinterval is $m_{\Delta}=\gamma (2N_{\Delta}+1)$, $L_{\Delta}=Tm_{\Delta}=T\gamma (2N_{\Delta}+1)$. The number of nodes on the overall interval $[a,b]$ is $M=Km_{\Delta}$. The algorithm's computational complexity is divided into the following parts:
\begin{itemize}
  \item Generation of matrix ${\bf F}_{\gamma, N_{\Delta}}^T$ and its SVD decomposition, the computational complexity is  $$\mathcal{O}\left(N^3_{\Delta}\right)\sim\mathcal{O}(1).$$
  It should be noted that because ${\bf F}_{\gamma, N_{\Delta}}^T$ is fixed in the method, the calculation of this step can be completed in advance and stored.
  \item Calculation of coefficient $ {\bf c}_{N_{\Delta},k}$, the amount of calculation is
$$C_{\Delta}\times m_{\Delta}\times K=C_{\Delta}M$$
where $C_{\Delta}$ is the number of ${\bm \sigma}_j$ greater than $\epsilon$ and we have
$$C_0\leq 2N_{\Delta}+1.$$
So the  computational complexity is
$$\mathcal{O}(M).$$
\item Obtain the extended data of each sub-interval, which can be achieved through IFFT, and the computational complexity is
$$K\times O\left(L_{\Delta}\log(L_{\Delta})\right)\sim  O\left(M\right).$$
\end{itemize}

Combining the previous test results, we actually hope that the product $T\gamma N_{\Delta}$ is minimized, so we recommend
 using the  parameters
\begin{equation}
  \gamma=1, T=6, N_{\Delta}=9
\end{equation}
for actual calculations. In this case
$$m_{\Delta}=\gamma\times (2N_{\Delta}+1)=19, L_{\Delta}=T\times m_{\Delta}=114.$$
\section{Numerical Examples\label{SEC4}}
In this subsection, some numerical experiments are given to verify the proposed method. All the tests were computed on a Windows 10 system with a 16 GB memory,
Intel(R) Core(TM)i7-8500U CPU\@1.80GHz using MATLAB 2016b.

{\bf Example 1} We take the test function as
$$f_1(x)=\frac{1}{1+25x^2},\quad x\in [-1,1].$$
 Fig. \ref{6a} displays the distribution of absolute approximation errors at ten sampling points using partition points  $a_k \in \{-1+\frac{2k}{K}\}$. Key observations reveal: At  $K=4$,  errors in boundary-adjacent regions approach machine precision {with progressive accuracy degradation toward interior regions. This spatial pattern reflects the function's frequency characteristics:
$f_1(x)$  exhibits smoother behavior near boundaries compared to higher-frequency in interior regions. Doubling partition density to  $K=8$  leaves the two central subintervals above machine precision threshold, while $K=12$  enables all subintervals to achieve errors approaching machine precision. This progression demonstrates that strategically positioned partition points informed by a priori knowledge of function oscillation frequency can achieve machine-level accuracy with sparser discretization. Supporting evidence appears in  Fig. \ref{6b}, comparing errors for non-uniform partitions $a_k\in\{0,\pm 0.2,\pm0.5,\pm 1\}$.  While adaptive sampling techniques could be incorporated to further optimize partition placement, such extensions fall beyond our current scope.
\begin{figure}
			\begin{center}
\subfigure[\label{6a}$a_k\in \{-1+\frac{2k}{K}\}$] {
\resizebox*{5.5cm}{!}{\includegraphics{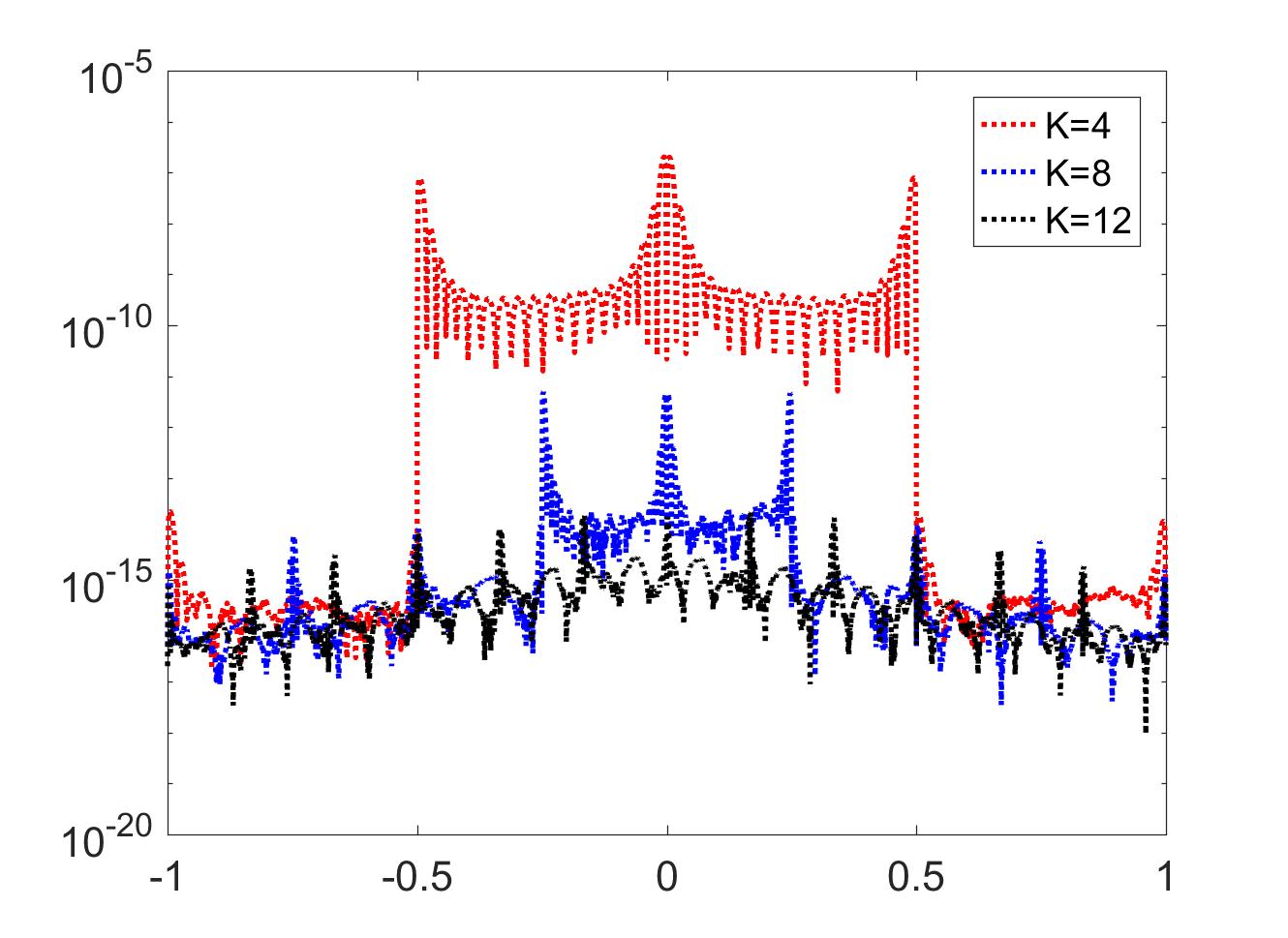}}
}%
\subfigure[\label{6b}$a_k\in\{0,\pm 0.2,\pm0.5,\pm 1\}$] {
\resizebox*{5.5cm}{!}{\includegraphics{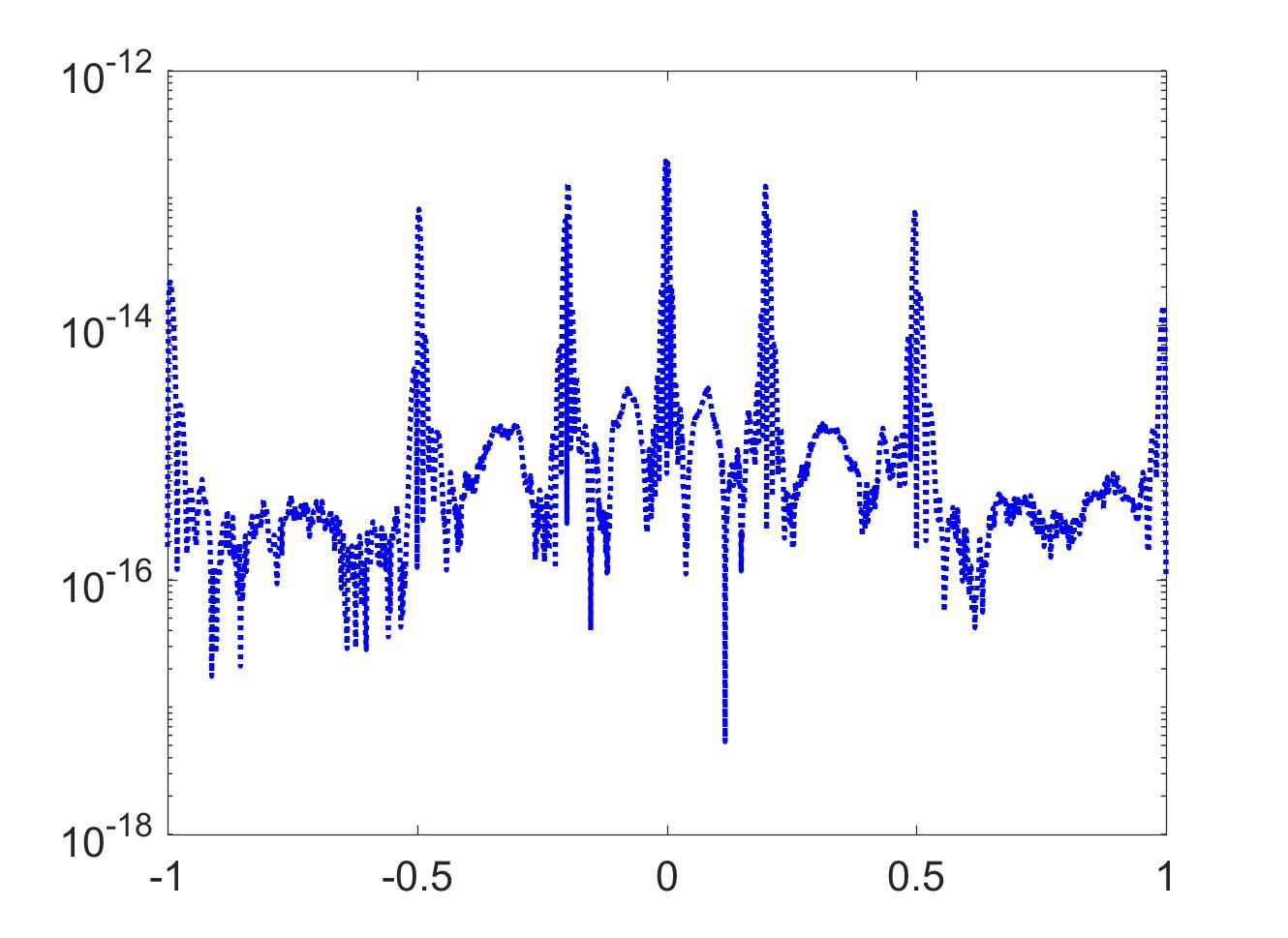}}
}%

		{\caption{The distribution of the absolute value of the approximation error. \label{ex1}}}
	\end{center}
\end{figure}

{\bf Example 2} Next, we take the test functions as
$$f_2(x)=\cos(200x^2),\quad f_3(x)=\text{Ai}(-66-70 x),\quad f_4(x)=e^{\sin(65.5\pi x-27\pi)-\cos(20.6\pi x)}.$$
where $\text{Ai}(\cdot)$ denotes the Airy function. Figs. \ref{7a}--\ref{7c} display the plots of these functions, revealing their pronounced high-frequency oscillatory behavior. Fig. \ref{7d} illustrates the evolution of the approximation error for the three functions as $K$ increases. Notably, as long as $K$ exceeds a value related to the frequency of the function, the algorithm can obtain an approximation close to machine accuracy, demonstrating the method's capability to resolve high-frequency features with near-optimal numerical accuracy.

\begin{figure}
			\begin{center}
\subfigure[\label{7a}$f_2(x)$] {
\resizebox*{5.5cm}{!}{\includegraphics{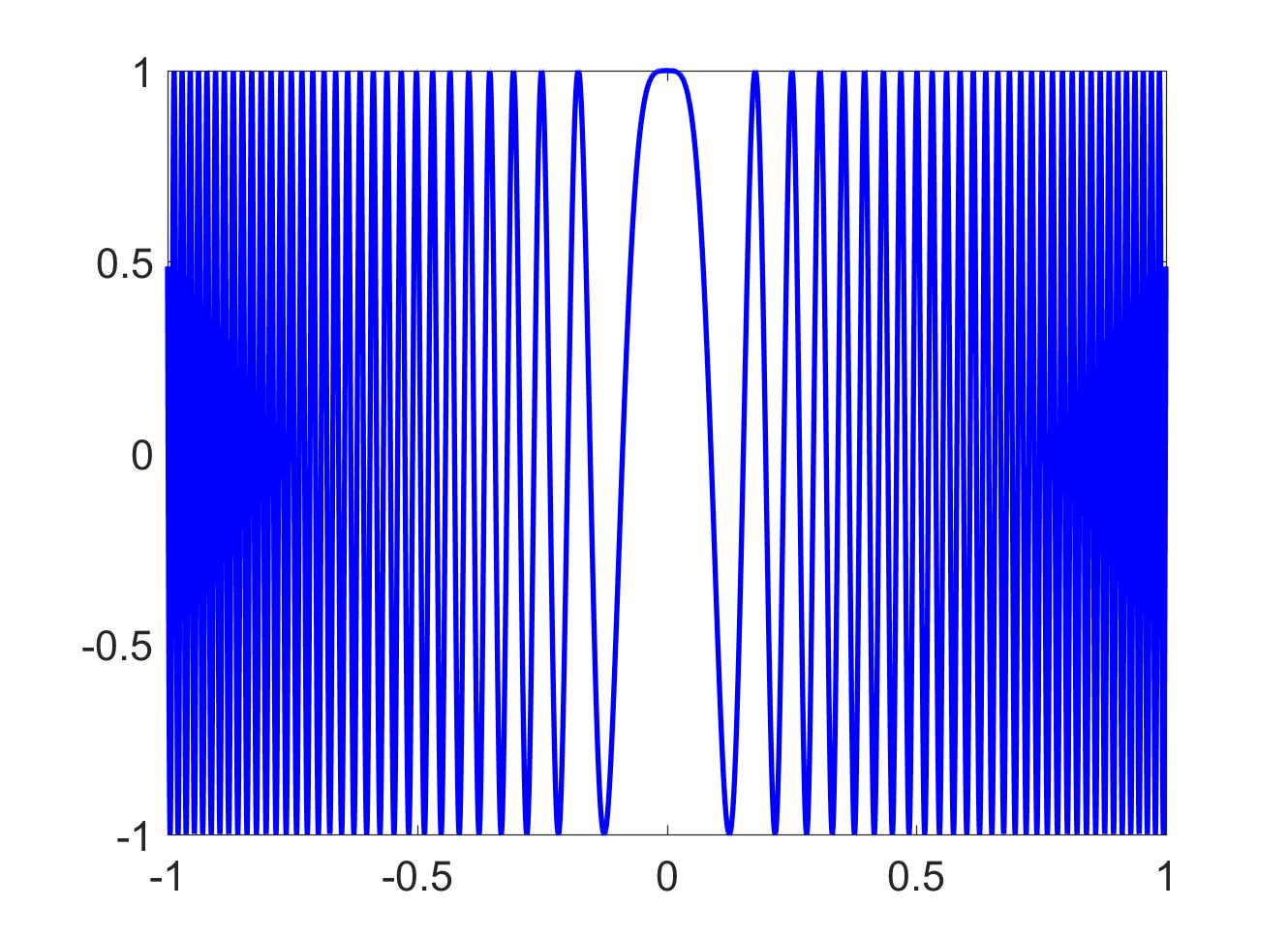}}
}%
\subfigure[\label{7b}$f_3(x)$] {
\resizebox*{5.5cm}{!}{\includegraphics{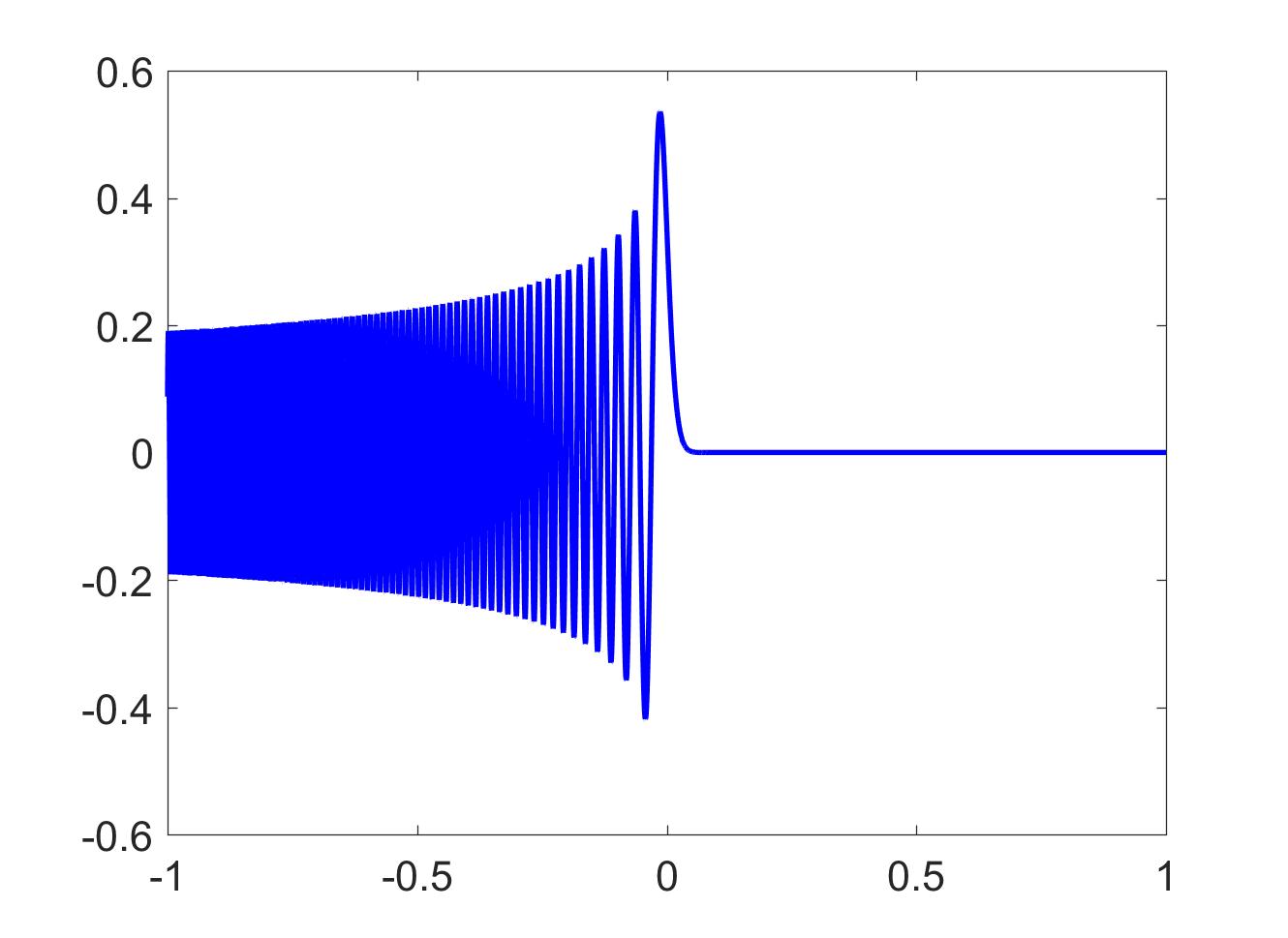}}
}%

\subfigure[\label{7c}$f_{4}(x)$] {
\resizebox*{5.5cm}{!}{\includegraphics{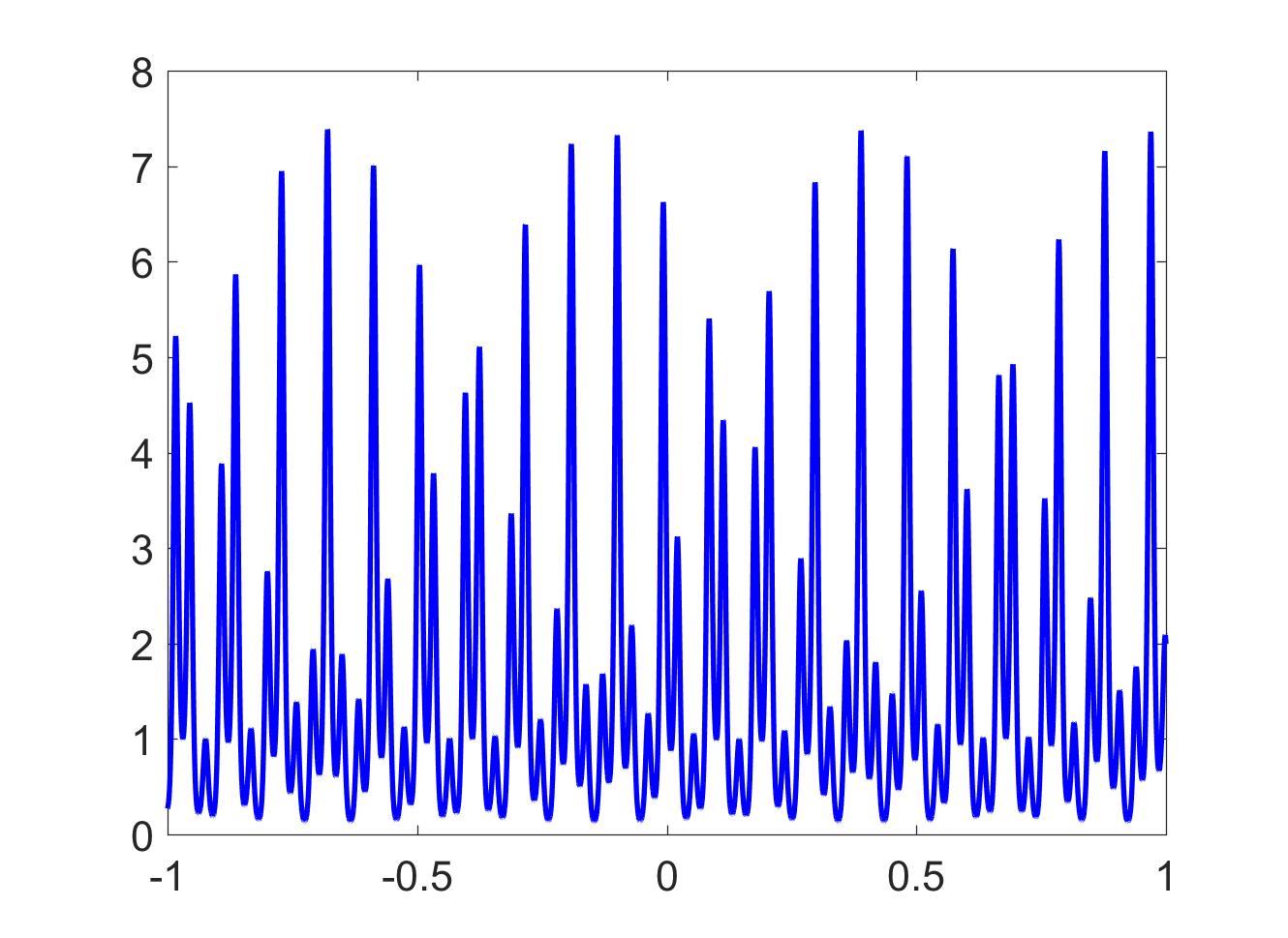}}
}%
\subfigure[\label{7d}approximation error] {
\resizebox*{5.5cm}{!}{\includegraphics{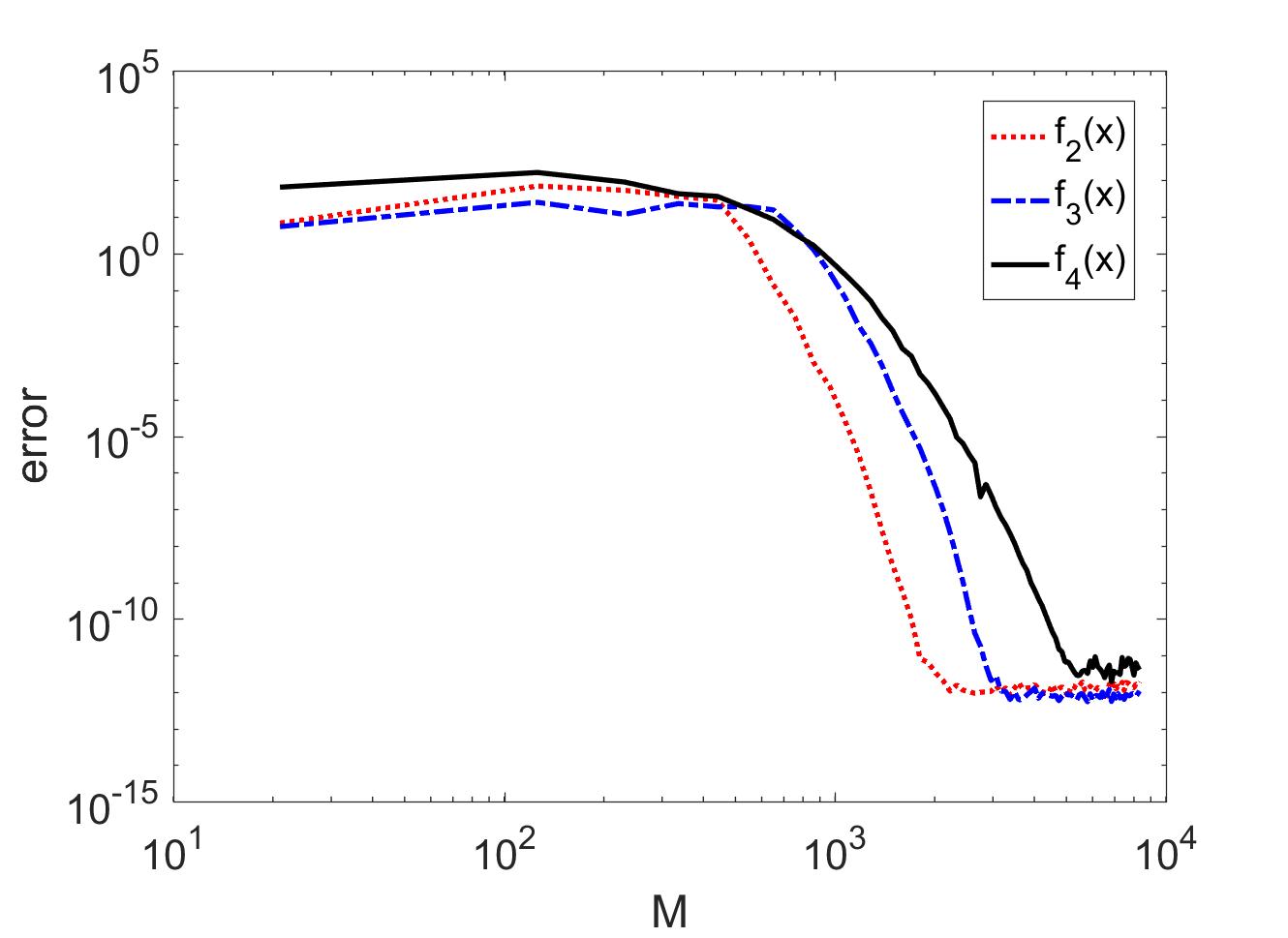}}
}%
		{\caption{The functions $f_2$, $f_3$, $f_4$ and their approximation error against $M$. \label{ex2}}}
	\end{center}
\end{figure}

{\bf Example 3} In numerous computational problems, employing a consistent discretization scale to approximate diverse functions is often essential. To demonstrate this capability, we maintain a fixed discretization level at
$K=20$ and evaluate the following representative test functions:
$$f_5(x)=e^{\sin(2.7\pi x)+\cos(\pi x)},\quad f_6(x)=x^2\sin(10 x),\quad f_7(x)=\frac{1}{8-7x}.$$
 Figs. \ref{8a}-\ref{8c} present the graphical representations of these functions, while  Fig. \ref{8d} compares the approximation errors generated by our method. The results demonstrate that the proposed approach achieves comparable approximation accuracy across these functionally distinct cases under identical discretization parameters. This consistent performance underscores the methodological versatility of our framework in handling diverse functional forms, as evidenced by the uniform error profiles despite varying function characteristics.

\begin{figure}
			\begin{center}
\subfigure[\label{8a}$f_5(x)$] {
\resizebox*{5.5cm}{!}{\includegraphics{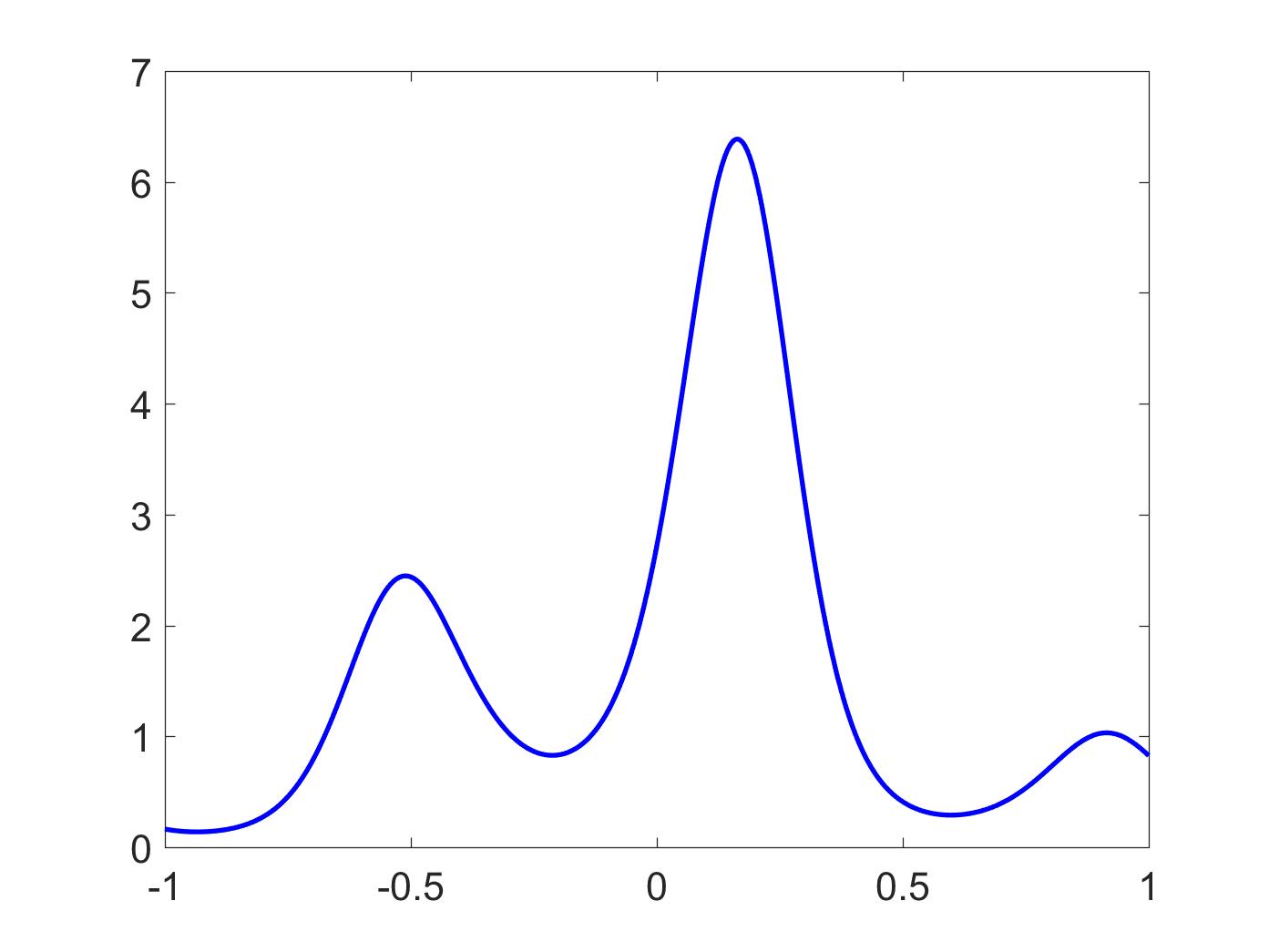}}
}%
\subfigure[\label{8b}$f_6(x)$] {
\resizebox*{5.5cm}{!}{\includegraphics{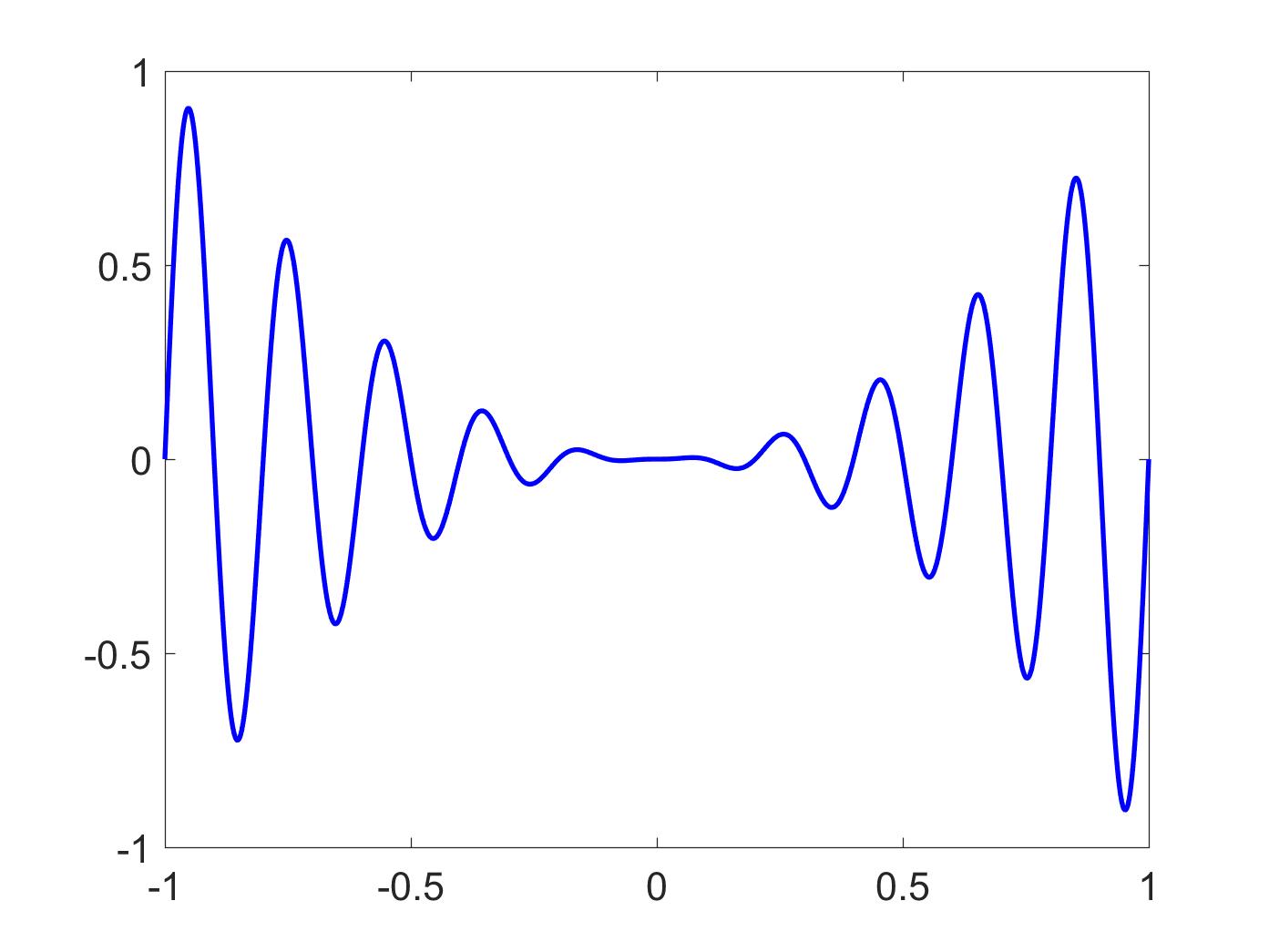}}
}%

\subfigure[\label{8c}$f_{7}(x)$] {
\resizebox*{5.5cm}{!}{\includegraphics{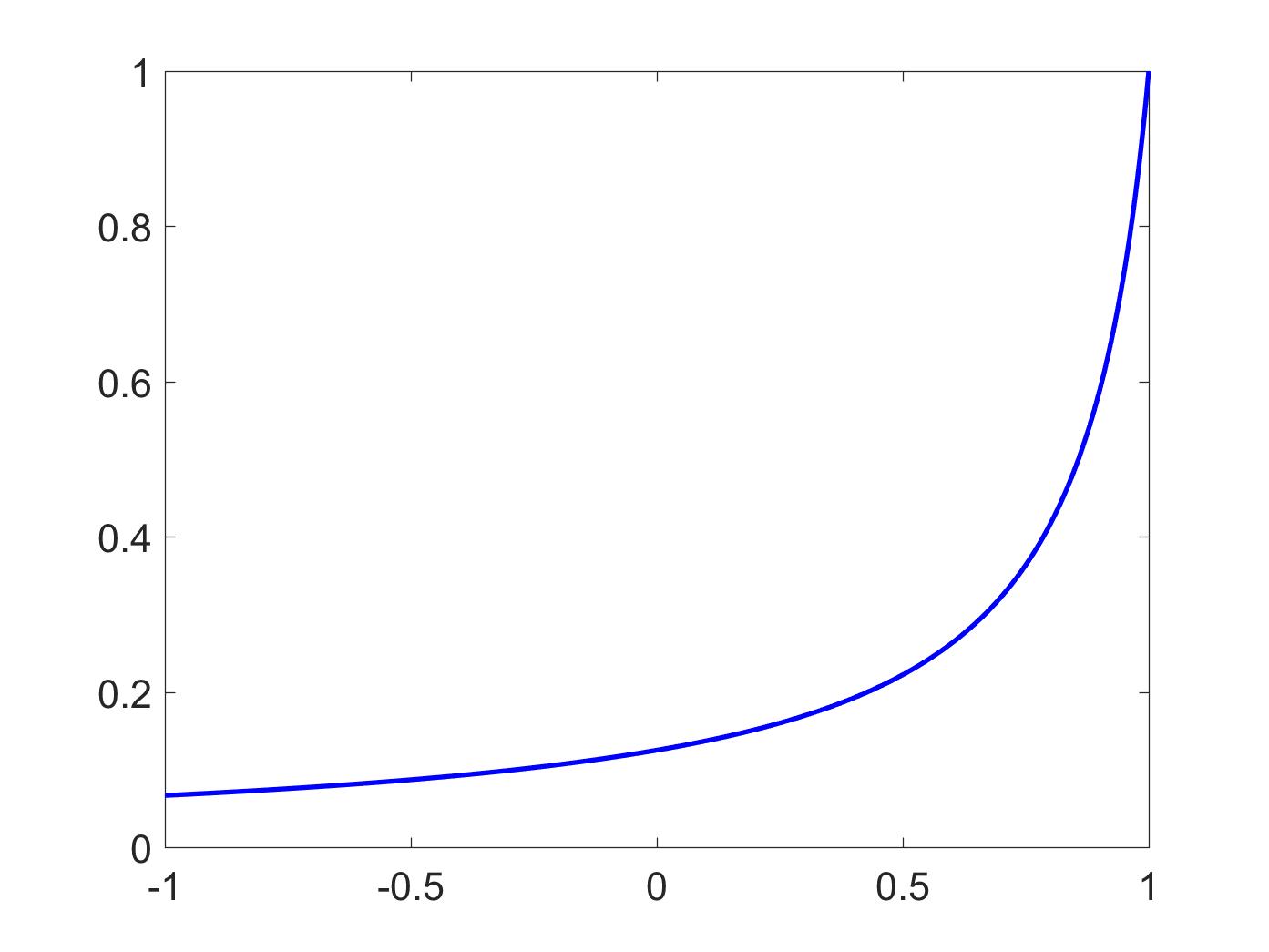}}
}%
\subfigure[\label{8d}approximation error] {
\resizebox*{5.5cm}{!}{\includegraphics{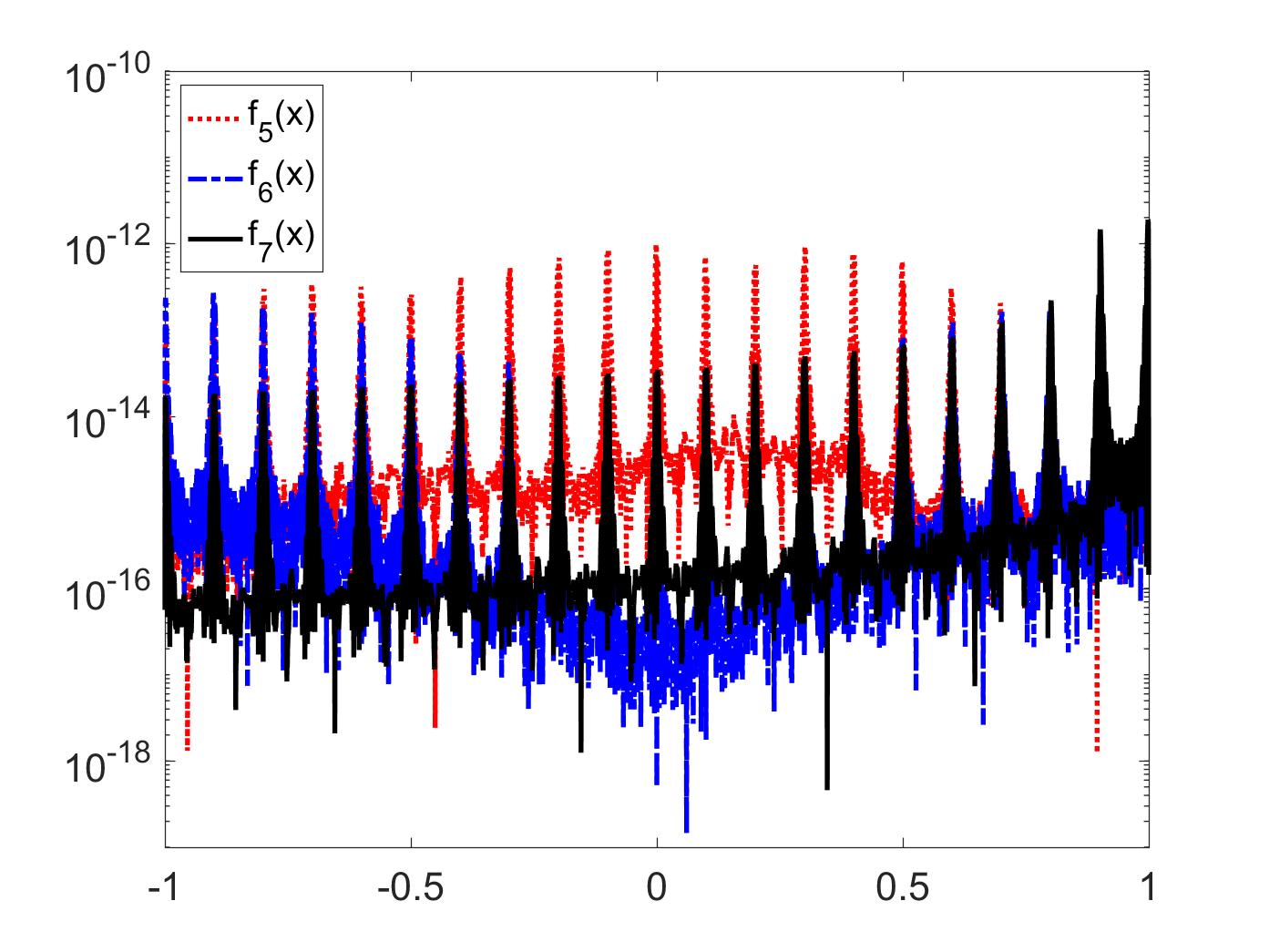}}
}%
		{\caption{The functions $f_5$, $f_6$, $f_7$ and their approximation error distribution. \label{ex3}}}
	\end{center}
\end{figure}

{\bf Example 4}.
In this example, we investigate a test function with localized singular behaviour, defined as
\begin{equation*}
  f_8(x)=
  \begin{cases}
    1, & -1\leq x\leq -\frac{1}{2},\\[2mm]
    \sin(\pi x), & -\frac{1}{2}<x\leq 0,\\[2mm]
    x^2, & 0<x\leq 1.
  \end{cases}
\end{equation*}
It is clear that $f_8$ is globally continuous; however, its first derivative is discontinuous at $d_1=0$, and the second derivative is discontinuous at $d_2=-\frac12$.

\begin{figure}[htbp]
  \centering
\subfigure[\label{9a} $f_8(x)$]{
\includegraphics[width=5.5cm]{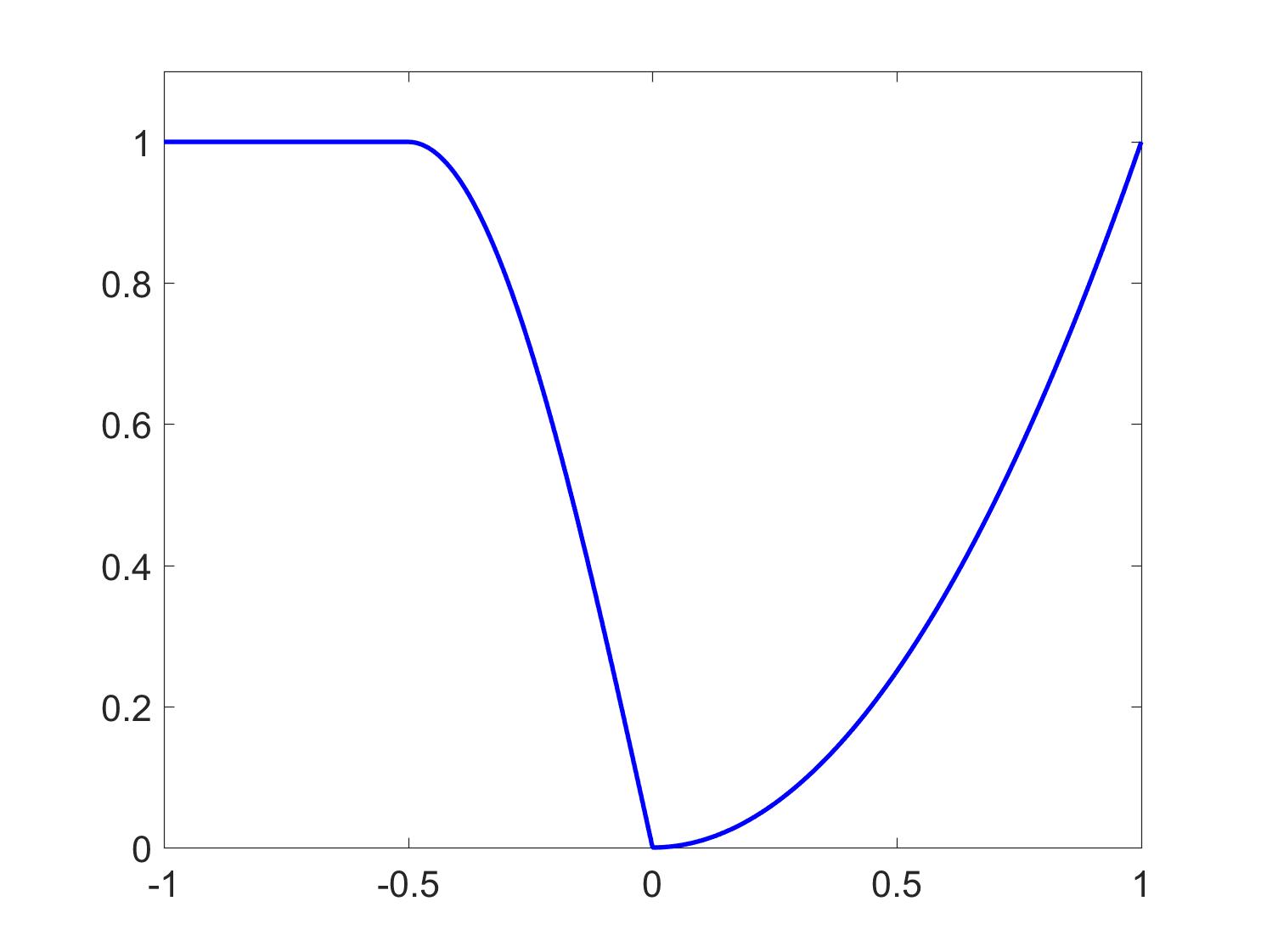}}
\subfigure[\label{9b} Approximation error for various $K$]{
\includegraphics[width=5.5cm]{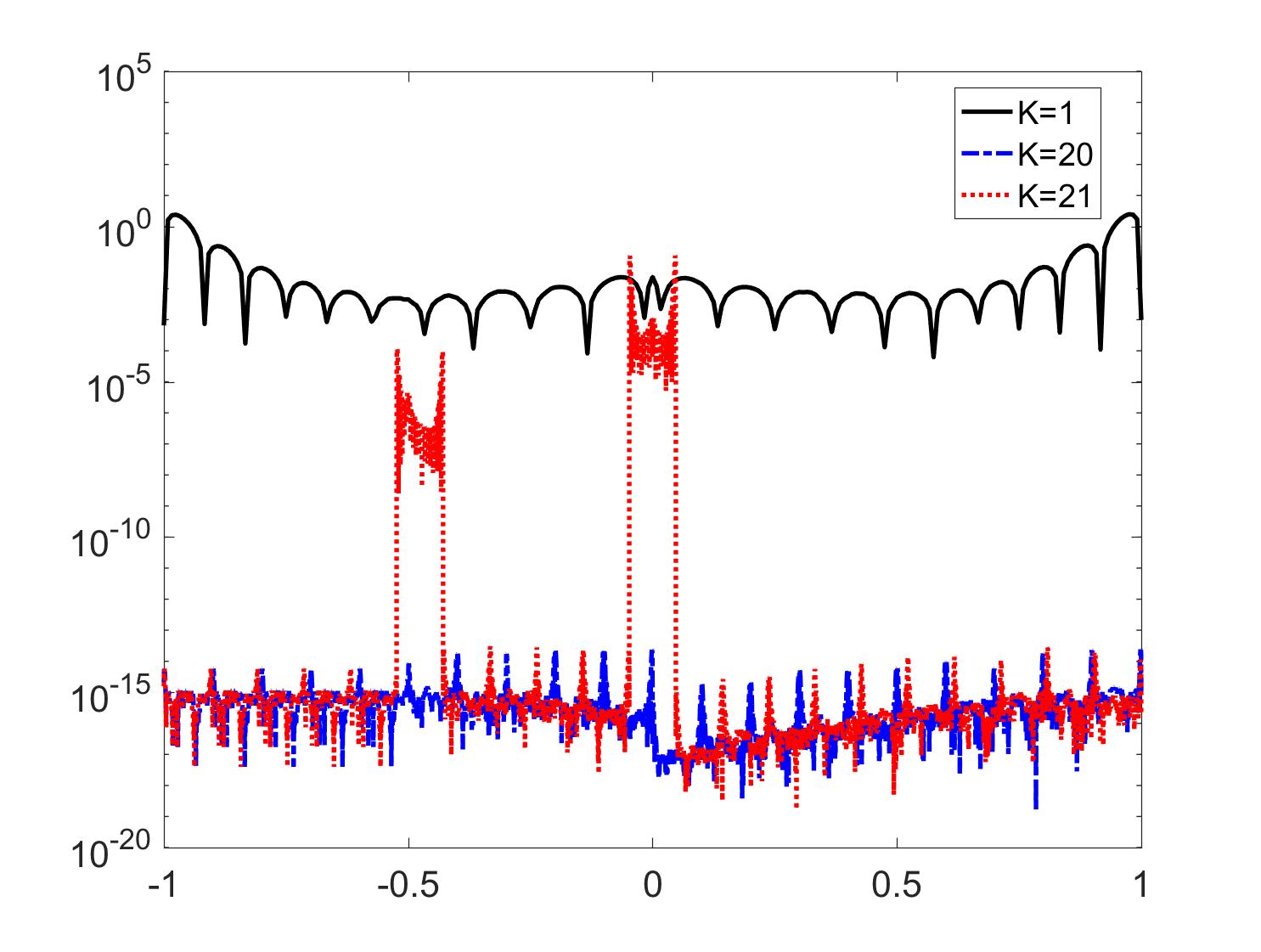}}
\subfigure[\label{9c} Identified singular intervals and singular points]{
\includegraphics[width=5.5cm]{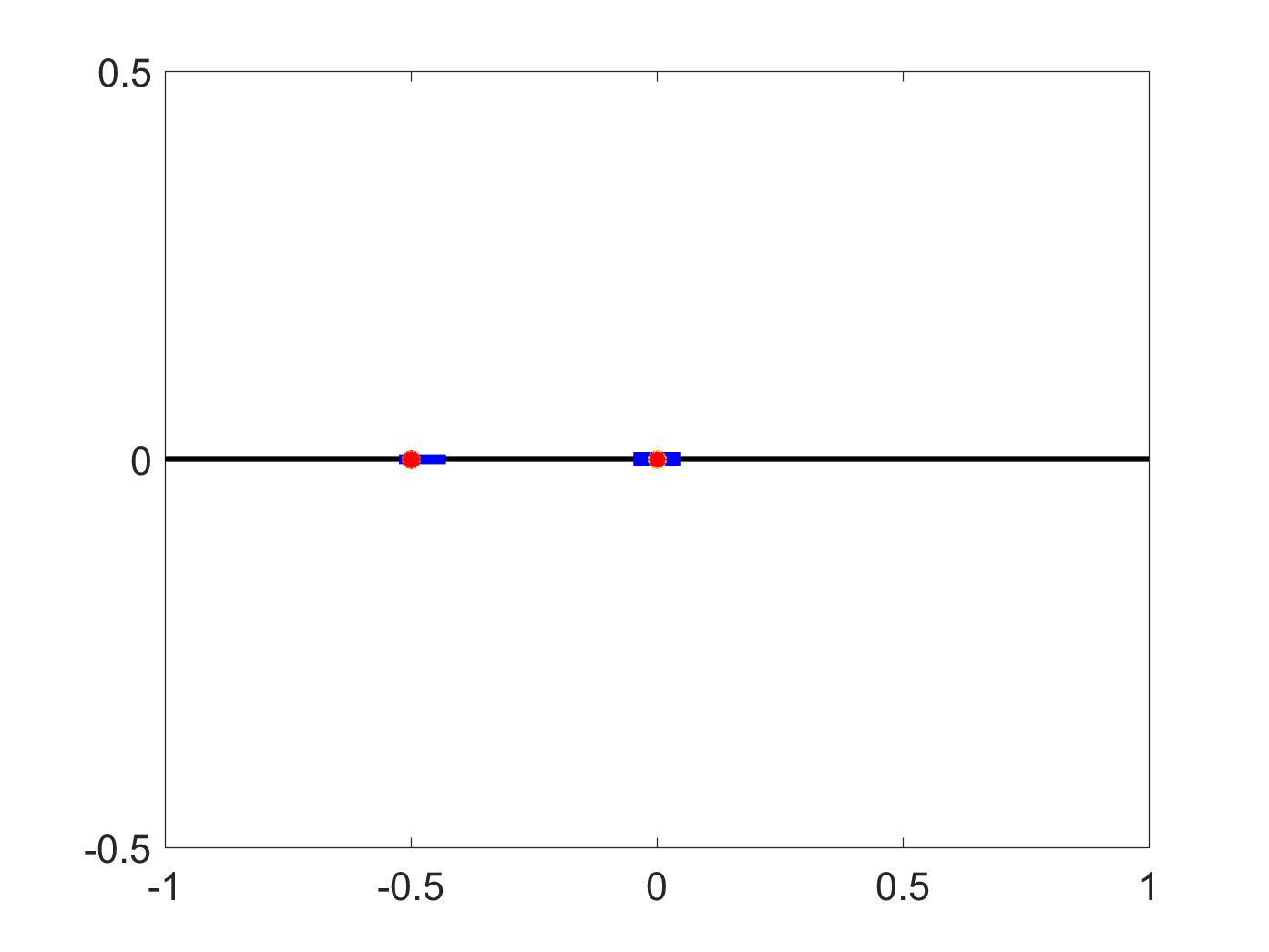}}
\subfigure[\label{9d} Approximation error after singular point correction ($K=21$)]{
\includegraphics[width=5.5cm]{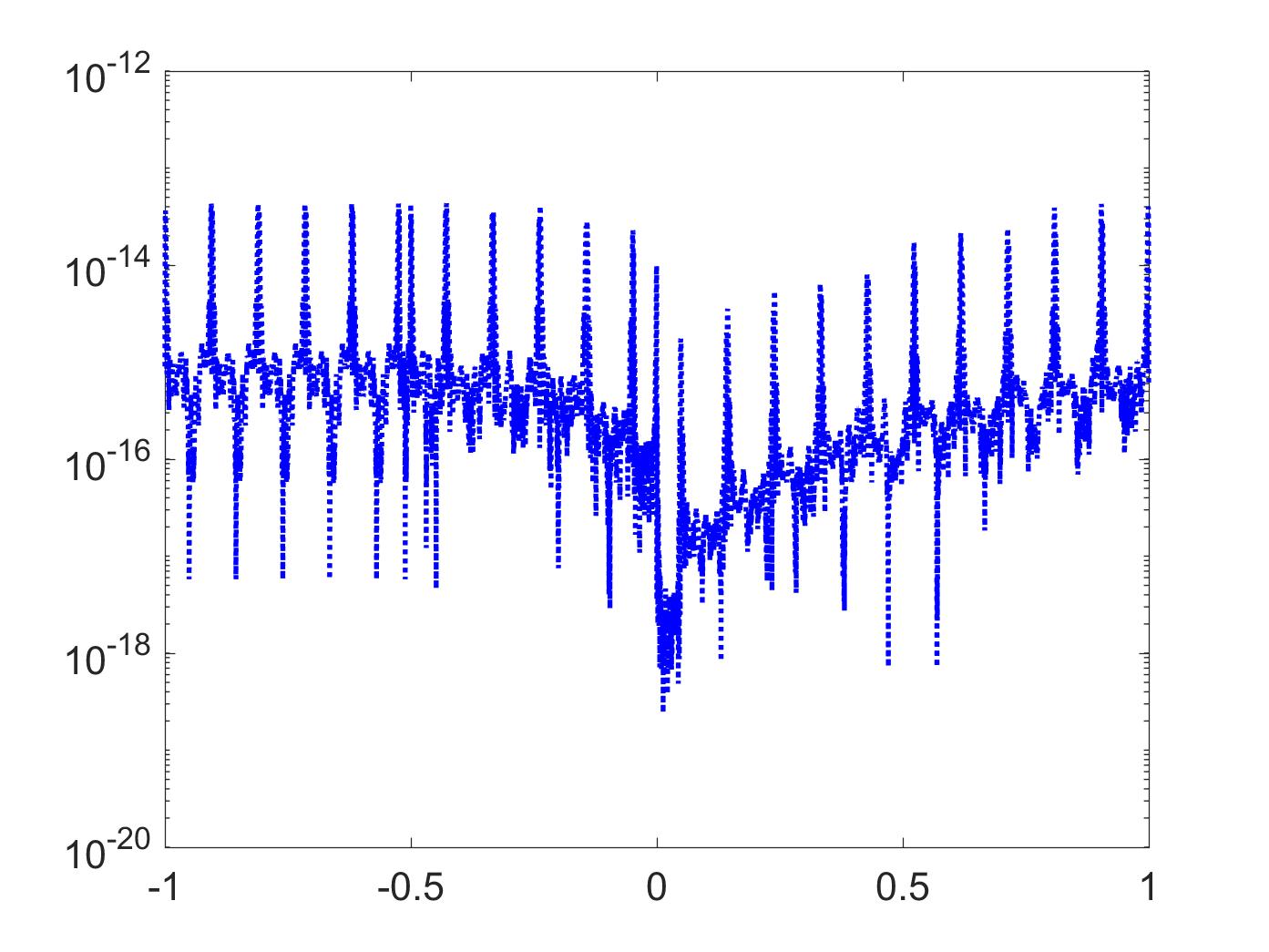}}
\caption{Numerical results for Example 4.}
\label{ex4}
\end{figure}

Figure~\ref{9b} reports the approximation error for $K=1$, $20$, and $21$. When $K=1$, the presence of singular points seriously deteriorates the accuracy over the entire interval. When $K=20$, both $d_1$ and $d_2$ lie in the partition points $\{a_k\}$, so each subinterval $I_j$ is free of singularities and spectral accuracy is achieved throughout. However, for $K=21$, $d_1$ and $d_2$ fall inside subintervals $I_{11}$ and $I_6$, respectively, resulting in a severe degradation of accuracy on those subintervals.

This observation raises a natural question: for cases like $K=21$, can we automatically localize the subintervals and positions of the singularities and then correct the approximation accordingly? Note that in general $g_k\notin R(\mathcal{F}_N)$, and thus the equation $\mathcal{F}_N{\bf c}_{N,k}=g_k$ is perturbed. If the $k_0$-th subinterval contains a singularity, $g_{k_0}$ suffers a large perturbation, and taking the regularization parameter $\epsilon$ too small yields a large coefficient norm $\|{\bf c}_{N,k_0}\|_2$. Table~\ref{tabex4} lists $\|{\bf c}_{N,k}\|_2$ for $K=21$. It is observed that $\|{\bf c}_{N,6}\|_2$ and $\|{\bf c}_{N,11}\|_2$ are markedly larger than the others, which serves as a reliable indicator for locating singularity-containing subintervals.

Once the suspected subinterval $I_{k_0}$ is located, we further subdivide it by defining
\begin{equation*}
IL_{k_0,i}=[x_{k_0-1,i},x_{k_0,i}],\qquad IR_{k_0,i}=[x_{k_0,i},x_{k_0+1,i}], \qquad i=1,\ldots,m-2.
\end{equation*}
Since $IL_{k_0,i}$ and $IR_{k_0,i}$ inherit the same equispaced nodes as the original grid, we may solve \eqref{discreteeq} again to obtain coefficients ${\bf cL}_{N,i}$ and ${\bf cR}_{N,i}$. We then select the index
\begin{equation*}
  i_0=\mathop{\arg\min}_{i}\bigl\{\|{\bf cL}_{N,i}\|_2+\|{\bf cR}_{N,i}\|_2\bigr\},
\end{equation*}
and take $x_{k_0,i_0}$ as the detected breakpoint. The function is then approximated using ${\bf cL}_{N,i_0}$ on $[x_{k_0,0},x_{k_0,i_0}]$ and using ${\bf cR}_{N,i_0}$ on $[x_{k_0,i_0},x_{k_0,m-1}]$. Figure~\ref{9c} displays the localization results, while Figure~\ref{9d} shows that, after singularity correction, spectral accuracy is effectively restored.

\begin{table}[htbp]
\centering
\caption{The values of $\|{\bf c}_{N,k}\|_2$ for $K=21$.}
\label{tabex4}
\small
\begin{tabular*}{\textwidth}{@{\extracolsep\fill}cccccccc}
\toprule
$k$ & 1 & 2 & 3 & 4 & 5 & 6 & 7 \\
\midrule
$\|{\bf c}_{N,k}\|_2$ & 9.74e1 & 9.74e1 & 9.74e1 & 9.74e1 & 9.74e1 & {\color{blue}3.24e8} & 8.63e1 \\
\bottomrule
\\[-2mm]
\toprule
$k$ & 8 & 9 & 10 & 11 & 12 & 13 & 14 \\
\midrule
$\|{\bf c}_{N,k}\|_2$ & 7.39e1 & 5.60e1 & 3.65e1 & {\color{blue}6.56e11} & 2.73e{-}1 & 5.72e{-}1 & 1.03e1 \\
\bottomrule
\\[-2mm]
\toprule
$k$ & 15 & 16 & 17 & 18 & 19 & 20 & 21 \\
\midrule
$\|{\bf c}_{N,k}\|_2$ & 1.66e1 & 2.45e1 & 3.43e1 & 4.58e1 & 5.90e1 & 7.40e1 & 9.10e1 \\
\bottomrule
\end{tabular*}
\end{table}

\section{Conclusions and remarks\label{SEC5}}
We conclude the paper with some final  remarks.
\begin{itemize}
  \item Methodological Framework: We present a novel localized Fourier extension scheme that achieves piecewise low-order approximations through domain segmentation. The framework establishes an analytical relationship between critical parameters $T$ and $\gamma$, accompanied by an optimized parameter selection strategy. This formulation ultimately yields a spectral approximation algorithm demonstrating $\mathcal{O}(M)$
computational complexity while maintaining exponential convergence rates.
      \item  Generalization Potential: The proposed localization methodology extends naturally to polynomial frame approximation and other frame systems. Future investigations will systematically compare their relative merits and limitations against our current framework, particularly regarding approximation efficiency, stability thresholds, and computational overhead.
  \item Methodological Transfer: The developed framework demonstrates inherent adaptability to broader problem classes through appropriate operator transformations. Establishing generalized implementation protocols for diverse computational scenarios will constitute a primary focus of forthcoming research, particularly in partial differential equation solutions and high dimensional approximation tasks. This approach demonstrates promising potential for gravity and magnetic data extrapolation\cite{WangIP2011}, with comprehensive analyses to be addressed in forthcoming research.
\end{itemize}


\begin{thebibliography}{10}

\bibitem{adcock2019frames}
B.~Adcock and D.~Huybrechs.
\newblock Frames and numerical approximation.
\newblock {\em SIAM Review}, 61(3):443--473, 2019.

\bibitem{adcock2020frames}
B.~Adcock and D.~Huybrechs.
\newblock Frames and numerical approximation {II}: generalized sampling.
\newblock {\em Journal of Fourier Analysis and Applications}, 26(6):87, 2020.

\bibitem{adcock2014numerical}
B.~Adcock, D.~Huybrechs, and J.~Mart{\'\i}n-Vaquero.
\newblock On the numerical stability of {Fourier} extensions.
\newblock {\em Foundations of Computational Mathematics}, 14:635--687, 2014.

\bibitem{bruno2010high}
O.~P Bruno and M.~Lyon.
\newblock High-order unconditionally stable {FC-AD} solvers for general smooth
  domains {I}. basic elements.
\newblock {\em Journal of Computational Physics}, 229(6):2009--2033, 2010.

\bibitem{Bruno2022}
O.~P. Bruno and J.~Paul.
\newblock Two-dimensional {Fourier} continuation and applications.
\newblock {\em SIAM Journal on Scientific Computing}, 44(2):A964--A992, 2022.

\bibitem{cohen2011wavelets}
J.~Cohen and A.~I. Zayed.
\newblock {\em Wavelets and Multiscale Analysis: Theory and Applications}.
\newblock 2011.

\bibitem{herremans2024efficient}
A.~Herremans and D.~Huybrechs.
\newblock Efficient function approximation in enriched approximation spaces.
\newblock {\em IMA Journal of Numerical Analysis}, page drae017, 2024.

\bibitem{Huybrechs2010}
D.~Huybrechs.
\newblock On the {Fourier} extension of nonperiodic functions.
\newblock {\em SIAM Journal on Numerical Analysis}, 47:4326--4355, 2010.

\bibitem{Lyon2011}
M.~Lyon.
\newblock A fast algorithm for {Fourier} continuation.
\newblock {\em SIAM Journal on Scientific Computing}, 33(6):3241--3260, 2011.

\bibitem{Matthysen2016FAST}
R.~Matthysen and D.~Huybrechs.
\newblock Fast algorithms for the computation of {Fourier} extensions of
  arbitrary length.
\newblock {\em SIAM Journal on Scientific Computing}, 38(2):A899--A922, 2016.

\bibitem{plonka2018numerical}
G.~Plonka, D.~Potts, G.~Steidl, and M.~Tasche.
\newblock {\em Numerical {Fourier} Analysis}.
\newblock 2018.

\bibitem{shen2011spectral}
J.~Shen, T.~Tang, and L.~L. Wang.
\newblock {\em Spectral Methods: Algorithms, Analysis and Applications},
  volume~41.
\newblock Springer, 2011.

\bibitem{WangIP2011}
Y.~F. Wang, I.~E. Stepanova, V.~N. Titarenko, and A.~G. Yagola.
\newblock {\em Inverse Problems in Geophysics and Solution Methods}.
\newblock Higher Education Press, Beijing, 2011.

\bibitem{Webb2020}
M.~Webb, V.~Copp\'{e}, and D.~Huybrechs.
\newblock Pointwise and uniform convergence of {Fourier} extensions.
\newblock {\em Constructive Approximation}, 52:139--175, 2020.

\bibitem{zhao2024fast}
Z.~Y. Zhao, Y.~F. Wang, and A.~G. Yagola.
\newblock Fast algorithms for {Fourier} extension based on boundary interval
  data.
\newblock {\em Numerical Algorithms}, 2025.

\bibitem{zhao2025new}
Z.~Y. Zhao, Y.~F. Wang, A.~G. Yagola, and X.~S. Li.
\newblock A new approach for {Fourier} extension based on weighted generalized
  inverse.
\newblock {\em arXiv preprint arXiv:2501.16096}, 2025.

\end{thebibliography}
\end{document}